\documentclass[11pt]{amsart}
\usepackage{amsmath,amssymb,epsfig,color, amsthm, mathrsfs}
\newtheorem{theorem}{Theorem}[section]

\newtheorem{corollary}[theorem]{Corollary}
\newtheorem{definition}[theorem]{Definition}
\newtheorem{lemma}[theorem]{Lemma}
\newtheorem{proposition}[theorem]{Proposition}
\newtheorem{remark}[theorem]{Remark}

\newcommand{\vanish}[1]{}\parskip=12pt

\def\p{\prime}

\def\R{\mathcal{R}}
\def\T{\mathcal{T}}

\def\C{\mathcal{C}}
\def\D{\mathcal{D}}

\def\H{\mathcal{H}}
\newcommand{\VPE}{\operatorname{VP}}
\numberwithin{equation}{section}

\begin{document}

\title[Relative Tutte Polynomials of Tensor Products]
{Relative Tutte polynomials of tensor products of colored graphs}
\author{Y. Diao and G. Hetyei}
\address{Department of Mathematics and Statistics, UNC Charlotte,
    Charlotte, NC 28223}
\email{ydiao@uncc.edu, ghetyei@uncc.edu}
\date{\today}

\begin{abstract}
The tensor product $(G_1,G_2)$ of a graph $G_1$ and a pointed graph
$G_2$ (containing one distinguished edge) is obtained by identifying
each edge of $G_1$ with the distinguished edge of a separate copy of
$G_2$, and then removing the identified edges. A formula to compute the
Tutte polynomial of a tensor product of graphs was originally given by
Brylawski. This formula was recently generalized to colored graphs and the
generalized Tutte polynomial introduced by Bollob\'as and Riordan. In
this paper we generalize the colored tensor product formula to relative
Tutte polynomials of relative graphs, containing zero edges to which
the usual deletion-contraction rules do not apply. As we have shown in a
recent paper, relative Tutte polynomials may be used to compute the
Jones polynomial of a virtual knot.
\end{abstract}

\maketitle
\section{Introduction}\label{s1}

The Tutte polynomial is one of the most important invariants in graph
theory. It was first introduced and studied by Tutte for non-colored
graphs, but has since been generalized to colored graphs \cite{BR} and
to colored relative graphs in which some edges cannot be treated as
regular colored edges in the computation of the Tutte polynomial
\cite{DHRel}. The corresponding Tutte polynomial in the latter case is
called the relative Tutte polynomial.

There are many situations in applied graph theory where an actual network
is represented by a graph, whose edges turn out to denote subnetworks at
closer inspection. A typical example is an electric circuit whose
components are (identical) integrated circuits themselves. Theoretically,
we may represent many such networks of subnetworks by using the {\em
tensor product operation} of graphs. The tensor product operation associates
a graph $G_1\otimes G_2$ to a graph $G_1$ and a pointed graph $G_2$
containing one distinguished edge $e$. It is obtained by replacing
each edge of $G_1$ with a copy of $G_2\setminus e$, where $e$ is a
used to mark the vertices of $G_2$ where we graft $G_2$ to the place of
the removed edge of
$G_1$ and is itself removed in the process. The Tutte polynomial of such a
tensor product of graphs was first expressed by Brylawski~\cite{B}. He
showed that the Tutte polynomial of a tensor product can be obtained
from the Tutte polynomial $G_1$ and some ordinary and pointed Tutte
polynomials associated to $(G_2,e)$ through certain variable substitutions.
The application of the initial tensor product is limited by the fact
that its definition
requires all edges to be replaced by the same graph. However, in the network
setting, the components of an actual network may be integrated circuits of different kinds. Such a composite network cannot be obtained by replacing every connection with the same subnetwork (as required by the tensor product definition). Thus it is more practical and applicable to color the edges (links) of $G_1$
(a network) first and then replace only edges of a fixed color
with the same graphs (subnetworks) $G_2$ ($G_2\setminus e$ to be more precise). Repeating this
operation would then allow replacing individual edges by different
graphs. This new tensor product concept was introduced in \cite{DHCl} where it was shown that the results of
Brylawski \cite{B} on the Tutte polynomial of a tensor
product of non-colored graphs can be generalized to the (generalized)
Tutte polynomial of a (generalized) tensor product of colored
graphs \cite{DHCl}.

Another application of the colored Tutte polynomial is in the area of knot theory. It is well-known that the Jones polynomial of a
link can be computed from the Kauffman bracket polynomial while the Kauffman bracket polynomial of a link can be computed from the
(signed) Tutte polynomial of the face graph of a regular
projection of the link. This was first shown for alternating links and
the ordinary Tutte polynomial by Thistlethwaite~\cite{T0}, then
generalized to arbitrary links and a signed Tutte polynomial by
Kauffman~\cite{K2}. This enables applications of the
ordinary Tutte polynomials and their signed generalizations to
classical knot theory such as those in \cite{DEZ,DGH1,Ja}.
For virtual knots the situation is a little more complicated.
An appropriate generalization of the Kauffman bracket polynomial
was developed by Kauffman himself~\cite{K3}. However, until very
recently, no appropriate generalization of the Tutte polynomial to face
graphs of virtual links was known. In a series of papers, Chmutov,
Pak and Voltz \cite{Ch,CP,CV} developed a generalization of
Thistlethwaite's theorem first to checkerboard-colorable \cite{CP} then
to arbitrary \cite{Ch,CV} virtual link diagrams. These express the
Jones polynomial of a virtual link in terms of a signed generalization
of the Bollob\'as-Riordan polynomial~\cite{BR2,BR3} of a ribbon graph,
obtained from the virtual link diagram. In \cite{DHRel}, it is shown that
a relative variant of the {\em other} generalization of the Tutte
polynomial, also due to Bollob\'as and Riordan~\cite{BR} may also be
used to compute the Jones polynomial of a virtual link, this time
directly from the face graph of the virtual link diagram. In a face graph of a virtual link diagram, edges corresponding to virtual crossings cannot be treated as a regular edge and are called {\em zero edges} in \cite{DHRel}. The Tutte polynomial of a colored graph with zero edges generalized in \cite{DHRel} is called a {\em relative Tutte polynomial}.

Given a colored graph $G_1$ and a pointed colored graph $G_2$ such that
both may contain zero edges, their tensor product can be defined just as
in the case of two colored graphs, so long as the edges in $G_1$ to be
replaced by copies of $G_2$ are not zero edges and the distinguished
edge marked in $G_2$ for the gluing purpose is not a zero edge
either. The main goal of this paper is to formulate the relative Tutte
polynomial of the tensor product of two colored graphs with zero edges,
using only the Tutte polynomials obtained from $G_1$ and $G_2$ and
certain substitution rules. As it turned out, we have to generalize the
pointed polynomials used in the colored tensor product case, define new
pointed polynomials and introduce a set of much more complicated
substitution rules. Given the complexity level of the relative Tutte
polynomial, this should not be a surprise. It is actually somewhat
surprising that such a formulation still exists!

This paper is organized as follows.
In Section~\ref{s2} we review the relative Tutte polynomial and, in Section~\ref{s21} we introduce the concept of the universal relative Tutte polynomial. Section \ref{s3} contains the definition of our pointed universal relative Tutte
polynomials. These include the ones generalized from the pointed Tutte polynomials used previously and three new pointed Tutte polynomials. In Section \ref{sec:cs} we discuss the contracting sets in a tensor product of colored graphs with zero edges. Section \ref{sec:tp} contains our main result: the generalization of the tensor product formula to colored relative graphs. The concluding Section \ref{conc} contains a sample application of our main result and a few further remarks.

\section{A review of the relative Tutte polynomial}\label{s2}

In this section we review the notion of the relative Tutte polynomial of a
colored graph $G$, with respect to a set of edges $\H\subset
E(G)$ introduced in~\cite{DHRel}. We observe that the
results in~\cite{DHRel} may be easily generalized to the situation where the
edges in $\H$ do not all belong to the same color set. We also introduce
 the {\em universal relative Tutte polynomial} of a colored graph.

\begin{definition}\label{TD1}
Let $G$ be a graph with edge set $E(G)$ and let  $\H\subseteq E(G)$. A
subset $\C$ of $E(G)\setminus \H$ is called a
{\em contracting set} of $G$ with respect to $\H$ if $\C$ contains no
cycles and $E(G)\setminus(\C\cup \H)$ contains no cocycles. Given a
contracting set $\C$, the set $E(G)\setminus(\C\cup \H)$ is called the
corresponding {\em deleting set} and it is denoted by $\D$.
\end{definition}

Recall that the cocycles of a graph are its minimal sets of edges whose
removal increases the number of connected components. Sometimes we will
refer to $\C$, $\D$ and $\H$ as {\em graphs}, by which we mean the
subgraphs of $G$ induced by the respective set of edges.

\begin{definition}{\em
Let $G$ be a graph and $\H$ be a subset of $E(G)$. A {\em proper
  labeling} or {\em relative labeling} of the edges of $G$ with respect
to $\H$ is a map $\phi:\ E(G)\longrightarrow \mathbb{N}$ such that
$\H=\{e\in E(G)\::\: \phi(e)=0\}$ and the restriction of $\phi$
to $E(G)\setminus \H$ is an injective map into $\mathbb{Z}_+$.  We
say that $e_1$ is larger than $e_2$ if $\phi(e_1)>\phi(e_2)$. Let $\C$
be a contracting set of $G$ with respect to $\H$, then
\begin{itemize}
\item[a)] an edge $e\in \C$ is called {\em internally active} if $\D\cup \{e\}$
contains a cocycle $D_0$ in which $e$ is the smallest edge,
otherwise it is {\em internally inactive}.
\item[b)] an edge $f\in \D$ is called {\em externally active} if $\C\cup \{f\}$
contains a cycle $C_0$ in which $f$ is the smallest edge,
otherwise it is {\em externally inactive}.
\end{itemize}
}\end{definition}

As noted in~\cite[Remark 3.12]{DHRel}, activities of regular edges may be given in the following
equivalent definition.

\begin{definition}\label{acti_alt_def}{\em
Let $G$ be a graph and $\H$ be a subset of $E(G)$ and that a proper
  labeling $\phi$ has been given. Let $\C$
be a contracting set of $G$ with respect to $\H$, then
\begin{itemize}
\item[a)] an edge $e\in \C$ is internally
active if it becomes a bridge once all edges in $\D$
larger than $e$ are deleted, otherwise it is internally inactive;
\item[b)] an edge $e\in \D$ is externally active if it becomes a
loop after all edges in $\C$ larger than $f$ are contracted,
otherwise it is externally inactive.
\end{itemize}
}\end{definition}
The above equivalent definition depends of the following description of
contracting and deleting sets.
\begin{lemma}
\label{l:descend}
Let $G$ be a graph, let $\H$ be a subset of $E(G)$, and let $\phi$ be a proper
labeling. Let $\C\subseteq E(G)\setminus \H$ be a set of regular edges
and let $\D=E(G)\setminus (\C\cup \H)$. Then $\C$ is a contracting set
and $\D$ is the corresponding deleting set if and only if the following
holds for regular edge $e\in E(G)\setminus \H$ after contracting all
  edges in $f\in \C$ and all deleting edges $g\in\D$ satisfying
  $\phi(f)\geq \phi(e)$  and $\phi(g)\geq \phi(e)$:
\begin{enumerate}
\item If $e\in \C$ then $e$ does not become a loop;
\item if $e\in \D$ then $e$ does not become a bridge.
\end{enumerate}
\end{lemma}
The proof is straightforward and left to the reader. As a consequence of
Lemma~\ref{l:descend} we may find each contracting set $\C$, together
with the corresponding deleting set $\D$ by going through the list of
regular edges in the order of their labels and deciding to put each of
them either into $\C$ or into $\D$, contracting or deleting them
accordingly, subject only to the rules that we are not allowed to
contract a loop or delete a bridge.

The definition of a relative Tutte polynomial involves
contracting all edges in $\C$ and deleting all edges in $\D$. We perform
these operations in decreasing order of the labels. The resulting
graph $\H_\C$ contains only zero edges and will be replaced with a graph
invariant $\psi(\H_\C)$. The graph $\H_\C$ depends on the order of
deletions and contractions determined by
the proper labeling $\phi$. However, the multiset of
blocks of $\H_\C$ is independent of the order in which the
deletions and contractions are performed, see~\cite[Lemma 3.14]{DHRel}.
That's why we want the operator $\psi$ to be a {\em block invariant}
(see~\cite[Definition 3.13]{DHRel}), most of the times.
For applications in knot theory a generalization of
block invariants was introduced in \cite{DHRel}: maps on isomorphism
classes on graphs that are invariant under {\em vertex pivots}. These
operations are defined as sequences of {\em vertex splicings} and
{\em vertex splittings}. A {\em vertex splicing} is an
operation that merges two disjoint graphs by picking a vertex from each
and identifying these selected vertices, thus creating a cutpoint.
The opposite operation is {\em vertex splitting} that creates two disjoint
graphs by replacing a cutpoint $v$ with two copies $v_1$ and $v_2$, and
makes each block containing $v$ contain exactly one of $v_1$ and $v_2$.
\begin{definition}
{\em Let $G$ be a graph that has a cutpoint $u$. A {\em vertex pivot} is
  a sequence of vertex splittings and vertex splicings, of the following
  kind. First
  we split $G$ by creating two copies of $u$ and two disjoint graphs
  $G_1$ and $G_2$. Then we take a vertex $v_1\in V(G_1)$ from the
  connected component of $u_1$ and a vertex $v_2\in V(G_2)$ in the
  connected component of $u_2$ and we merge $G_1$ and $G_2$ by
  identifying $u_1$ with $u_2$.
}
\end{definition}
As noted in ~\cite[Section 4]{DHRel}, $\H_\C$ will be the
same up to performing a sequence of vertex pivots, independently of $\phi$.

Let $G$ be a graph and $\H\subseteq E(G)$. In~\cite{DHRel} a coloring
$c: E(G)\setminus \H\rightarrow \Lambda $ of the regular edges to a color set
$\Lambda$ was considered. However, the definitions and results stated
in~\cite{DHRel} may be generalized without any substantial change
to the situation where we color all edges of $G$, including the zero
edges, using a map $c: E(G)\rightarrow \Lambda$. Let us call a graph
$G$, together with such a coloring $c: E(G)\rightarrow \Lambda$ a {\em
  $\Lambda$-colored graph}. We may fix a subset
$\Lambda_0\subseteq \Lambda$ and require all edges in $\H$ to be
with colors from $\Lambda_0$.  The subgraph $\H$ is thus also a
$\Lambda_0$-colored graph.

\begin{definition}
{\em We call two $\Lambda_0$-colored graphs $\Gamma$ and $\Gamma'$ {\em
    vertex pivot equivalent} if $\Gamma'$ is isomorphic to a graph
  obtained from $\Gamma$ by performing a sequence of vertex pivot operations.
We call an invariant $\psi$ of $\Lambda_0$-colored graphs a {\em
    vertex pivot invariant} if $\psi(\Gamma)=\psi(\Gamma')$ whenever
$\Gamma$ and $\Gamma'$ are vertex pivot equivalent. The
collection of vertex pivot equivalence classes of $\Lambda_0$-colored
graphs is denoted by $\VPE(\Lambda_0)$.
}
\end{definition}

For any contracting set  $\C$ of $G$ with
respect to $\H$, let $\H_\C$ be the graph obtained by deleting all
edges in $\D$ and contracting all edges in $\C$. Finally, we assign a
proper labeling $\phi$ to the edges of $G$. We now define the relative
Tutte polynomial of $G$ with respect to $\H$ and $\psi$ as
\begin{equation}\label{eq1}
T_\H^\psi(G)=\sum_{\C} \left(\prod_{e\in G\setminus
\H}w(G,c,\phi,\C,e)\right)\psi(\H_\C)\in \R[x_\lambda, X_{\lambda},
  y_{\lambda}, Y_{\lambda} \::\: \lambda\in \Lambda],
\end{equation}
where the summation is taken over all contracting sets $\C$
and $w(G,c,\phi,\C,e)$  is the {\em weight} of the edge $e$ with
respect to the contracting set $\C$, which is defined as
(assume that $e$ has color $\lambda$):

\begin{equation}
w(G,c,\phi,\C,e)=\left\{
\begin{array}{ll}
X_\lambda &\ {\rm if}\ e\ {\rm is\ internally\ active};\\
Y_\lambda &\ {\rm if}\ e\ {\rm is\ externally\ active};\\
x_\lambda &\ {\rm if}\ e\ {\rm is\ internally\ inactive};\\
y_\lambda &\ {\rm if}\ e\ {\rm is\ externally\ inactive}.
\end{array}
\right.
\end{equation}

To simplify our notation, we may use $T_\H(G)$ for
$T_\H^\psi(G)$, with the understanding that some $\psi$ has been
chosen, unless there is a need to stress what $\psi$ really is.
Following \cite{DHRel}, we then write
$$
W(G,c,\phi,\C)=\prod_{e\in G\setminus \H}w(G,c,\phi,\C,e)
$$
so that
\begin{equation}\label{eq2}
T_\H(G,\phi)=\sum_{\C} W(G,c,\phi,\C)\psi(\H_\C).
\end{equation}

One of our main results~\cite[Theorem 3.16]{DHRel} extends
the famous result of Bollob\'as and Riordan~\cite[Theorem 2]{BR}
on colored Tutte polynomials to colored relative Tutte polynomials.
Its proof extends without any change to the situation when the set of
zero edges is a $\Lambda_0$-colored subgraph.

\begin{theorem}\label{DHRel}
Assume $I$ is an ideal of $\R[x_\lambda, X_{\lambda},
  y_{\lambda}, Y_{\lambda} \::\: \lambda\in \Lambda]$. Then the
homomorphic image of $T_\H(G,\phi)$ in $\R[x_\lambda, X_{\lambda},
  y_{\lambda}, Y_{\lambda} \::\: \lambda\in \Lambda]/I$ is
independent of $\phi$ (for any $G$ and $\psi$) if and only if
\begin{equation}\label{Ieq1}
\det\left(\begin{array}{ll}
X_\lambda & y_\lambda\\
X_\mu & y_\mu
\end{array}\right)
- \det\left(\begin{array}{ll}
x_\lambda & Y_\lambda\\
x_\mu & Y_\mu
\end{array}\right)\in I
\end{equation}
and
\begin{equation}\label{Ieq2}
\det\left(\begin{array}{ll}
x_\lambda & Y_\lambda\\
x_\mu & Y_\mu
\end{array}\right)
- \det\left(\begin{array}{ll}
x_\lambda & y_\lambda\\
x_\mu & y_\mu
\end{array}\right)\in I.
\end{equation}
hold for all $\lambda,\mu \in \Lambda$.
\end{theorem}

Motivated by this result, we will assume $T_\H(G,\phi)$ is defined in
the ring
$$\T(\R,\Lambda):=\R[x_\lambda, X_{\lambda},
  y_{\lambda}, Y_{\lambda} \::\: \lambda\in \Lambda]/I_1(\R,\Lambda)$$
where $I_1(\R,\Lambda)$ is the ideal of $\R[x_\lambda, X_{\lambda},
  y_{\lambda}, Y_{\lambda} \::\: \lambda\in \Lambda]$ generated by all
polynomials of the form \eqref{Ieq1} and \eqref{Ieq2}.
\begin{definition}
We call the ring $\T(\R,\Lambda)$ the {\em Tutte ring} associated to the
color set $\Lambda$ and the ring of coefficients $\R$.
\end{definition}
The relative Tutte polynomial, considered as an element of the Tutte
ring $\T(\R,\Lambda)$,  becomes independent of the choice of the proper
labeling $\phi$ and we may write $T_H(G)$ for
$T_H(G,\phi)$.  An immediate consequence of this fact is the
following corollary, see~\cite[Corollary 3.17]{DHRel}.

\begin{corollary}\label{Cor2.7}
$T_\H(G)$ can be computed via the following recursive formula, valid for
  any regular edge $e$, i.e., any $e\not\in \H$:
\begin{equation}\label{recur1}
T_\H(G)=\left\{
\begin{array}{ll}
y_\lambda T_\H(G- e) + x_\lambda T_\H(G/e),&\mbox{
if $e$ is neither a bridge nor a loop,}\\
X_\lambda T_\H(G/e), & \mbox{if  $e$ is a bridge,}\\
Y_\lambda T_\H(G- e), & \mbox{if $e$ is a loop.}\\
\end{array}
\right.
\end{equation}
In the above, $e\not\in\H$ is a regular edge, $\lambda=c(e)$,
$G- e$ is the graph obtained from $G$ by deleting $e$
and $G/e$ is the graph obtained from $G$ by contracting $e$.
\end{corollary}

\begin{remark}
{\em
In some situations, it is plausible to require that the set of colors
used to color the regular edges be disjoint from the set
$\Lambda_0$ used to color the zero edges, i.e. $c(E(G\setminus \H))\subseteq
\Lambda\setminus \Lambda_0$, although we do not need this restriction in
what is written above. For the sake of convenience and to avoid possible confusions, we will assume that $c(E(G\setminus \H))\subseteq
\Lambda\setminus \Lambda_0$ in the rest of this paper.}
\end{remark}

\section{The universal relative Tutte polynomial}\label{s21}

We now introduce the {\em universal relative Tutte
  polynomial} associated to a color set $\Lambda_0$.
\begin{definition}
\label{def:urT}
{\em Let $G$ be a $\Lambda$-colored graph and $\H$ a $\Lambda_0$-colored
 subset of its edges such that  $\Lambda_0\subseteq \Lambda$ and $c(E(G\setminus \H))\subseteq
\Lambda\setminus \Lambda_0$. Let us
 introduce a distinct variable $z_{[\Gamma]}$ for each vertex pivot
 equivalence class $[\Gamma]\in\VPE(\Lambda_0)$. Let  $\psi_{\Lambda_0}$ be
 the vertex pivot invariant that assigns to each
  $\Lambda_0$-colored graph $\Gamma$ the variable $z_{[\Gamma]}$ in the
  polynomial ring $\R[z_{[\Gamma]}\::\: [\Gamma]\in
    \VPE(\Lambda_0)]$. We call the relative Tutte
  polynomial
$$T_\H^{\psi_{\Lambda_0}}(G)\in \R[x_\lambda,X_\lambda,y_{\lambda},
    Y_{\lambda}, z_{[\Gamma]}\::\: \lambda\in \Lambda\setminus
    \Lambda_0, [\Gamma]\in
    \VPE(\Lambda_0)]/I_1 ({\R},\Lambda,\Lambda_0)$$
the {\em universal $\Lambda_0$-colored relative Tutte polynomial of $G$
  with respect to $\H$} and denote it by $T_\H^{\Lambda_0}(G)$. Here
$I_1 (\R,\Lambda,\Lambda_0)$ is
the ideal of $\R [x_\lambda,X_\lambda,y_{\lambda}, Y_{\lambda},
    z_{[\Gamma]}\::\: \lambda\in \Lambda\setminus \Lambda_0, [\Gamma]\in
    \VPE(\Lambda_0)]$
  generated by all polynomials of the
form \eqref{Ieq1} and \eqref{Ieq2} with $\lambda,\mu\in \Lambda
\setminus \Lambda_0$. We
call the ring
$$
\T(\R,\Lambda,\Lambda_0):=\R[x_\lambda,X_\lambda,y_{\lambda},
    Y_{\lambda}, z_{[\Gamma]}\::\: \lambda\in \Lambda\setminus
    \Lambda_0, [\Gamma]\in \VPE(\Lambda_0)]/I_1 ({\R},\Lambda,\Lambda_0)
$$
the {\em $\Lambda_0$-pointed Tutte ring associated to the color set
  $\Lambda$ and the ring of coefficients $\R$.}
}
\end{definition}
The ideal $I_1 (\R,\Lambda,\Lambda_0)$ in Definition~\ref{def:urT} above
is generated by polynomials not containing any of the variables
    $z_{[\Gamma]}$. Thus we have
\begin{equation}
\label{eq:ptr}
\T(\R,\Lambda,\Lambda_0)=\T(\R,\Lambda\setminus
\Lambda_0)[z_{[\Gamma]}\::\: [\Gamma]\in \VPE(\Lambda_0)].
\end{equation}
In other words, the $\Lambda_0$-pointed Tutte ring
$\T(\R,\Lambda,\Lambda_0)$ is a polynomial ring in which
the Tutte ring $\T(\R,\Lambda\setminus \Lambda_0)$ is the ring of
coefficients and $\{z_{[\Gamma]}\::\: [\Gamma]\in \VPE(\Lambda_0)\}$ is
the set of independent variables and the universal relative Tutte polynomial $T_\H^{\psi_{\Lambda_0}}(G)$
is a special element in the Tutte ring $\T(\R,\Lambda,\Lambda_0)$, namely one
that is a $\T(\R,\Lambda\setminus \Lambda_0)$-linear combination of the terms of the form $z_{[\Gamma]}$.
This observation makes the substitution map,
given in Theorem~\ref{thm:univ} below, well-defined.
This theorem justifies the adjective {\em universal} in the name of the
universal $\Lambda_0$-pointed relative Tutte polynomial.  We call it a
theorem only because of its importance, its proof is straightforward.

\begin{theorem}
\label{thm:univ}
Let $G$ be a $\Lambda$-colored graph and $\H$ a $\Lambda_0$-colored
 subset of its edges where $\Lambda_0\subseteq \Lambda$ and $c(E(G\setminus \H))\subseteq
\Lambda\setminus \Lambda_0$.
Let $\psi$ be a vertex pivot invariant of $\Lambda_0$-colored graphs
with values in an integral domain $\R$. Then the homomorphism
$$
\T(\R,\Lambda,\Lambda_0)\rightarrow \T(\R,\Lambda\setminus\Lambda_0),
$$
sending each element of $\T(\R,\Lambda\setminus \Lambda_0)$ into itself
and sending each $z_{[\Gamma]}$ into $\psi(\Gamma)$, sends the universal
  $\Lambda_0$-colored Tutte polynomial $T_\H^{\Lambda_0}(G)$ into
the relative Tutte polynomial $T_\H^\psi(G)$.
\end{theorem}
\begin{remark}
{\em $\T(\R,\Lambda,\Lambda_0)$ is a polynomial ring with infinitely many
variables $z_{[\Gamma]}$. However, if we consider only colored graphs with at most $N$
edges, where $N$ is any positive integer, then it may be replaced
with a polynomial ring with only finitely many variables.
}
\end{remark}

\section{Pointed universal relative Tutte polynomials}\label{s3}

In analogy to the main results in~\cite[Theorem 5.1]{DGH1} and
\cite[Theorem 3]{DHCl}, we want to obtain a
formula for the universal relative  Tutte polynomial of a
$\lambda$-colored  tensor product $G_1 \otimes_\lambda G_2$ of two
$\Lambda$-colored graphs that have $\Lambda_0$-colored subsets of
zero edges $\H_1$ and $\H_2$ (we will assume $\Lambda_0\subset
\Lambda$ and $\lambda\in \Lambda\setminus \Lambda_0$). Similarly to the
formulation in~\cite{DGH1,DHCl}, our formula will make use of pointed
variants of the universal relative  Tutte polynomial. Two of these
variants will be generalizations of the polynomials $T_C(G,e)$ and $T_L(G,e)$
that were already introduced in \cite{DGH1,DHCl} and which are
generalizations of pointed Tutte polynomials introduced by
Brylawski~\cite{B,BO}. A third variant arises with the presence of zero edges.
We will also have to introduce two further pointed Tutte polynomials
which will assume the role played by $T(G/e)$ and by $T(G-e)$, respectively,
in~\cite{DGH1,DHCl}. As before, when considering the $\lambda$-colored
tensor product $G_1 \otimes_\lambda G_2$, all pointed relative Tutte polynomials
will be computed for the pointed graph $G_2$.

In this section we assume that $G$ is a pointed connected graph with a
distinguished edge $e$ which is neither a loop nor a bridge, and $\H$
is a subset of $E(G)$ not containing $e$. We assume that $G$ is a
$\Lambda\cup\{\nu\}$-colored graph where $\nu\not\in \Lambda$, $\H$ is
$\Lambda_0$-colored where $\Lambda_0\subset \Lambda$. The
distinguished edge $e$ is marked by the unique color $\nu\not\in \Lambda$ to avoid possible confusions. Denote the set $\Lambda\cup \{\nu\}$ by $\Lambda^\p$, $\Lambda_0\cup \{\nu\}$ by $\Lambda_0^\p$ and $\H\cup \{e\}$ by $\H^\p$. The first three pointed Tutte
polynomials to be introduced are homomorphic images of the
universal $\Lambda_0^\p$-colored relative Tutte polynomial $T_{\H^\p}^{\Lambda_0^\p}(G)$. When we calculate $T_{\H^\p}^{\Lambda_0^\p}(G)$, we consider the distinguished edge $e$ as a zero edge. Again, let us stress that the color $\nu$ assigned to $e$ is different from the colors of all other (regular or zero) edges.

We want to classify the pairs $(\C,\D)$ of contracting sets and
corresponding deleting sets with respect to $\H^\p$ into three classes,
depending on their relation to the distinguished edge $e$, as follows:
\begin{definition}
\label{def:types}
Let $G$ be a pointed graph with distinguished edge $e$ and set of zero
edges $\H^\prime=\H\cup\{e\}$. Let $\C$ be a contracting set
of $G$ with respect to $\H^\prime$ and let $\D$ be the corresponding
deleting set.
\begin{itemize}
\item[(i)] We say that $(\C,\D)$ has {\em type $\mathscr{C}$} if
  $\C\cup\{e\}$ contains a cycle;
\item[(ii)] We say that $(\C,\D)$ has {\em type
  $\mathscr{D}$} if $\D\cup\{e\}$ contains a cocycle;
\item[(iii)] We say that $(\C,\D)$ has {\em type
  zero} if it has neither type $\mathscr{C}$ nor type $\mathscr{D}$.
\end{itemize}
To simplify our terminology we will also say that a contracting set
$\C$, or a deleting set $\D$ has type $\mathscr{C}$, $\mathscr{D}$ or zero, if the unique
pair $(\C,\D)$ formed with the corresponding deleting or contracting set
has the same type.
\end{definition}
The choice of letters to denote the types may seem counter-intuitive in
this section, the motivation will become clear in Section~\ref{sec:cs}.
Notice that an equivalent condition for $(\C,\D)$ to be of type zero is that
$\C\cup\H^\p$ contains a cycle but $\C\cup\{e\}$ does not.
Furthermore, $(\C,\D)$ cannot have type $\mathscr{C}$ and type $\mathscr{D}$ simultaneously: if $e$ closes a cycle with $\C$ in $G$ then
after removing all edges of $\D$ and the edge $e$ from $G$,
the endpoints of $e$ are still connected via a path containing the
edges in $\C$, hence deleting $\D\cup\{e\}$ will not increase the number of connected components in $G$. Thus
another equivalent description of the three types may be stated as
follows.
\begin{proposition}
\label{prop:tequiv}
$(\C,\D)$ has type $\mathscr{C}$, $\mathscr{D}$, or zero,
respectively, if and
only if after contracting the edges of $\C$ and deleting the edges of
$\D$ in $G$, the edge $e$ becomes a loop, bridge, or
neither loop nor bridge, respectively. In the type zero case, after contracting the edges of $\C$ and deleting the edges of
$\D$ in $G$, there is a path consisting of zero edges only (in
$\H$) connecting the endpoints of $e$.
\end{proposition}

The next statements depend on, and also characterize the type of $(\C,\D)$.

\begin{proposition}
\label{prop:typeC}
If $(\C,\D)$ is of type $\mathscr{C}$ then $\C$ is a  contracting set with
respect to $\H$ but  $\C\cup\{e\}$ is not a contracting set.
\end{proposition}
\begin{proof}
Clearly $\C\cup\{e\}$ is not a contracting set since $\C$ contains a
path connecting the endpoints of $e$, and adding $e$ to this path
creates a cycle. As seen in the proof of Proposition~\ref{prop:local},
the set $\C$ does not contain any cycle. We only need to check that
$\D\cup \{e\}$ contains no cocycle in $G$.
This may be performed in perfect analogy to the
corresponding part in the proof of Proposition~\ref{prop:local}, the
only difference being that the path $\gamma'$ introduced in that proof
may now be replaced by the path in $\C$ connecting the endpoints of $e$.
\end{proof}
\begin{proposition}
\label{prop:typeD}
If $(\C,\D)$ is of type $\mathscr{D}$ then $\C\cup\{e\}$ is a contracting set with respect to $\H$ but
$\C$ is not a contracting set.
\end{proposition}
\begin{proof}
Deleting all edges of $\D\cup\{e\}$ from $G$
disconnects the endpoints of $e$. Thus the set $\D\cup\{e\}$ is not a
deleting set with respect to $\H$ and there is no path in $\C$ connecting the end points of $e$.
Equivalently, $\C$ is not a contracting set and $\C\cup \{e\}$ contains no cycle. The proof of the fact that $\D$
contains no cocycle of $G$ is identical to the corresponding
part of the proof of Proposition~\ref{prop:local}.
\end{proof}
\begin{proposition}
\label{prop:type0}
If $(\C,\D)$ is of type
zero then both $\C$ and $\C\cup\{e\}$ are contracting sets with
respect to $\H$.
\end{proposition}
\begin{proof}
Since $(\C,\D)$ is of type zero, $\C\cup \{e\}$ (hence $\C$)
contains no cycle, but there is a path consisting of edges of
$\C\cup\H$ connecting the end vertices of $e$. The set $\D$
contains no cocycle by Proposition~\ref{prop:local}. Since there is a
path consisting of edges of $\C\cup\H$ connecting the end vertices
of $e$, adding $e$ to $\D$ does not create a cocycle in $G_2^f$.
\end{proof}
\begin{remark}
\label{rem:types}
{\em Since the premises in Propositions~\ref{prop:typeC},
  \ref{prop:typeD} and \ref{prop:type0} mutually exclude each other,
  the conclusions provide a {\em characterization} of the types of
  $(\C,\D)$: it has type $\mathscr{C}$ exactly when
$\C$ is a  contracting set with respect to $\H$ but
$\C\cup\{e\}$ is not, type $\mathscr{D}$ exactly when
   $\C\cup\{e\}$ is a  contracting set with respect to $\H$ but
$\C$ is not, and it type zero exactly when both $\C$ and
  $\C\cup\{e\}$ are contracting sets with respect to $\H$.
}
\end{remark}

To define our pointed Tutte polynomials we introduce five endomorphisms
of the $\Lambda_0$-pointed Tutte ring
$\T(\R,\Lambda^\p,\Lambda_0^\p)$, which is a polynomial ring
by~\eqref{eq:ptr}. The restriction of each of
these endomorphisms to $\T(\R,\Lambda)$ will be the identity map, thus
they can be given by prescribing their effect on the variables
$\{z_{[\Gamma]}\::\: [\Gamma]\in \VPE(\Lambda_0^\p)\}$. The
first three maps, $\pi_{C}$, $\pi_{L}$ and $\pi_{0}$, leave $z_{[\Gamma]}$
unchanged for select types of graphs $\Gamma$ and they send all the
other $z_{[\Gamma]}$ into zero:
\begin{eqnarray*}
\pi_{C}(z_{[\Gamma]})&=&
\left\{
\begin{array}{rl}
z_{[\Gamma]}&\mbox{if $\Gamma$ has exactly one edge $f$ of color $\nu$ and
  $f$ is a bridge (coloop);}\\
0&\mbox{otherwise.}
\end{array}
\right.
\\
\pi_{L}(z_{[\Gamma]})&=&
\left\{
\begin{array}{rl}
z_{[\Gamma]}&\mbox{if $\Gamma$ has exactly one edge $f$ of color $\nu$ and
  $f$ is a loop;}\\
0&\mbox{otherwise.}
\end{array}
\right.
\\
\pi_{0}(z_{[\Gamma]})&=&
\left\{
\begin{array}{rl}
z_{[\Gamma]}&\mbox{if $\Gamma$ has exactly one edge $f$ of color $\nu$ and
  $f$ is neither a loop nor a bridge;}\\
0&\mbox{otherwise.}
\end{array}
\right.
\\
\end{eqnarray*}
The last two maps, $\pi_/$ and $\pi_-$ perform a contraction or deletion
on some graphs $\Gamma$, and send all other $z_{[\Gamma]}$ into zero:
\begin{eqnarray*}
\pi_{/}(z_{[\Gamma]})&=&
\left\{
\begin{array}{rl}
z_{[\Gamma/f]}&\mbox{if $\Gamma$ has exactly one edge $f$ of color $\nu$ and
  $f$ is not a loop;}\\
0&\mbox{otherwise.}
\end{array}
\right.
\\
\pi_{-}(z_{[\Gamma]})&=&
\left\{
\begin{array}{rl}
z_{[\Gamma-f]}&\mbox{if $\Gamma$ has exactly one edge $f$ of color $\nu$ and
  $f$ is not a bridge;}\\
0&\mbox{otherwise.}
\end{array}
\right.
\\
\end{eqnarray*}

\begin{definition}
\label{def:CL0}
We define the pointed universal $\Lambda_0$-colored relative Tutte
polynomials $T^{\Lambda_0^\p}_{\H,C}(G,e)$, $T^{\Lambda_0^\p}_{\H,L}(G,e)$ and
$T^{\Lambda_0^\p}_{\H,0}(G,e)$, respectively, as the image of $T_{\H^\p}^{\Lambda_0^\p}(G)$ under the endomorphism
$\pi_{/}\circ\pi_{C}$, $\pi_{-}\circ\pi_{L}$ and $\pi_{0}$,
respectively.
\end{definition}
Notice that
in the special case that $\H=\emptyset$, the
definitions of $T^{\Lambda_0^\p}_{\H,C}(G,e)$ and
$T^{\Lambda_0^\p}_{\H,L}(G,e)$ yield $T_C(G,e)\cdot
z_{[\bullet]}$ and $T_L(G,e)z_{[\bullet]}$, respectively. Here
$T_C(G,e)$ and $T_L(G,e)$ are the polynomials defined in
\cite{DGH1,DHCl} and $\bullet$ is the graph containing a single
vertex. Thus, in situations where there is no confusion about the sets $\Lambda$, $\Lambda_0$ and $\H$, we will simply use $T_C(G,e)$, $T_L(G,e)$ and
$T_0(G,e)$ as the abbreviations for $T^{\Lambda_0^\p}_{\H,C}(G,e)$, $T^{\Lambda_0^\p}_{\H,L}(G,e)$ and
$T^{\Lambda_0^\p}_{\H,0}(G,e)$ respectively. As a consequence of Theorem~\ref{DHRel}, the pointed universal
$\Lambda_0$-colored  relative Tutte polynomials defined above may be
computed by summing weights of contracting sets of $G$ with respect to
$\H^\p$. The weights will be assigned using a proper labeling with
respect to $\H^\p$, but the outcome will be independent of the
labeling. The following lemmas are direct consequences of the
definitions of $T_{C}$, $T_{L}$ and $T_{0}$.

\begin{lemma}
\label{l:TC}
A contracting set $\C$ of $G$ with respect to
$\H^\p$ contributes a zero term to $T_C(G,e)$ unless it has type $\mathscr{D}$.
\end{lemma}
\begin{lemma}
\label{l:TL}
A contracting set $\C$ of $G$ with respect to
$\H^\p$ contributes a zero term to $T_L(G,e)$ unless it
has type $\mathscr{C}$.
\end{lemma}
\begin{lemma}
\label{l:T0}
A contracting set $\C$ of $G$ with respect to
$\H^\p$ contributes a zero term to $T_0(G,e)$ unless it
has type zero.
\end{lemma}

Next we define the pointed universal relative Tutte polynomials which
will assume the roles played by $T(G/e)$ and by $T(G-e)$ respectively
in~\cite{DGH1,DHCl}.

\begin{definition}
\label{def:T/} We define the pointed universal relative Tutte polynomials
$T^{\Lambda_0^\p}_{\H,/}(G,e)$ and $T^{\Lambda_0^\p}_{\H,-}(G,e)$,
respectively, as $T^{\Lambda_0}_{\H}(G/e)-\pi_/T^{\Lambda_0^\p}_{\H,0}(G,e)$ and
$T^{\Lambda_0}_{\H}(G-e)-\pi_-T^{\Lambda_0^\p}_{\H,0}(G,e)$, respectively.
\end{definition}

\begin{remark}\label{r4.12} {\em Notice that although the type zero contracting sets of $G-e$ are exactly those type zero contracting sets that make non-zero contributions in $\pi_-T^{\Lambda_0^\p}_{\H,0}(G,e)$,  they may not make the same
contributions in $T^{\Lambda_0}_{\H}(G-e)$ and $\pi_-T^{\Lambda_0^\p}_{\H,0}(G,e)$. The reason is that in $T^{\Lambda_0}_{\H}(G-e)$,
the edge $e$ is removed first while in
$\pi_/T^{\Lambda_0^\p}_{\H,0}(G,e)$ the edge $e$ is removed last. This
means that the contributions of the type zero contracting sets to
$T^{\Lambda_0}_{\H}(G-e)$ may not cancel with the contributions of the
corresponding type zero contracting sets to
$\pi_/T^{\Lambda_0^\p}_{\H,0}(G,e)$. In general,
$T^{\Lambda_0^\p}_{\H,-}(G,e)$ may even contain negative terms. The graph on the left side of Figure \ref{T_example} shows such an example. We will leave it to our reader to verify that $T^{\Lambda_0}_{\H}(G-e)=X_\mu z_{[\Gamma_b]}$ and $\pi_-T^{\Lambda_0^\p}_{\H,0}(G,e)=x_\mu z_{[\Gamma_b]}$, here $\Gamma_b$ is the graph that consists of a single zero edge that is a bridge. Thus $T^{\Lambda_0^\p}_{\H,-}(G,e)=(X_\mu-x_\mu) z_{[\Gamma_b]}$.
The situation for $T^{\Lambda_0^\p}_{\H,/}(G,e)$ is similar. For the graph $G$ shown on the right side of Figure \ref{T_example}, we have
$T^{\Lambda_0}_{\H}(G/e)=Y_\mu z_{[\Gamma_l]}$ while $\pi_/T^{\Lambda_0^\p}_{\H,0}(G,e)=y_\mu z_{[\Gamma_l]}$, where $\Gamma_b$ is the graph that consists of a single zero loop edge. Thus  $T^{\Lambda_0^\p}_{\H,/}(G,e)=(Y_\mu-y_\mu) z_{[\Gamma_l]}$.}
\end{remark}
\begin{figure}[htbp]
\begin{center}
\input{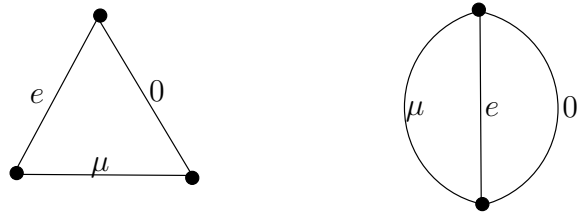}
\caption{Simple examples of graphs $G$ with the property that $T^{\Lambda_0^\p}_{\H,-}(G,e)$ (left) or $T^{\Lambda_0^\p}_{\H,/}(G,e)$ (right) contains negative terms.}\label{T_example}
\end{center}
\end{figure}

Of course, in the case when $\H=\emptyset$, we have
$T^{\Lambda_0^\p}_{\H,/}(G,e)=T(G/e)\cdot
z_{[\bullet]}$ and $T^{\Lambda_0}_{\H,-}(G,e)=T(G-e)\cdot
z_{[\bullet]}$. Again, in situations where there will be no confusion
about the sets $\H$, $\Lambda$ and $\Lambda_0$, we will use $T_/(G,e)$
and $T_-(G,e)$ as the abbreviations for $T^{\Lambda_0^\p}_{\H,/}(G,e)$
and $T^{\Lambda_0^\p}_{\H,-}(G,e)$ respectively.

We conclude this section with a generalization of \cite[Theorem 2]{DHCl}
and of its consequences. We will need this result to justify why the {\em
 regular Brylawski homomorphism}, to be introduced in Section~\ref{sec:tp}, is
well-defined.

\begin{theorem}\label{thm:pointed}
The pointed universal $\Lambda_0$-colored  relative Tutte
polynomials $T_C(G,e)$,
$T_L(G,e)$, $T_/(G,e)$, and
$T_-(G,e)$
satisfy the following two identities for
any $\mu\in \Lambda$:
\begin{eqnarray}
x_{\mu}\left(T_/(G,e)-
T_C(G,e)\right)
&=&
(Y_\mu-y_\mu) T_L(G,e),\label{TLC1}\\
y_\mu\left(T_-(G,e) -
T_L(G,e)\right)
&=& (X_\mu-x_\mu) T_C(G,e).\label{TLC2}
\end{eqnarray}
\end{theorem}

\begin{proof} \eqref{TLC1} is proved in a way that is analogous to the establishment of equation (5) in
the proof of \cite[Theorem 2]{DHCl}. Because of the presence of the zero edges, the proof is harder here and we choose to provide a detailed proof for our reader. By the definition of  $T_/(G,e)$, what we need to prove is
\begin{equation}\label{first}
x_{\mu}\left(T(G/e)-\pi_/T_{0}(G,e)-
T_C(G,e)\right)
=(Y_\mu-y_\mu) T_L(G,e).
\end{equation}
Notice that in order to compute each of the three polynomials on the
left side of (\ref{first}), we only need to consider type $\mathscr{D}$
contracting sets with respect to $\H'$. For the calculation of $T_C(G,e)$ this
 observation is stated in  Lemma~\ref{l:TC}. To calculate $T(G/e)$ and
$\pi_/T_{0}(G,e)$ we need to sum over contracting sets $\C$ of $G$ with
 respect to $\H^\prime$ that have the property that contracting all
 edges of $\C$ does not turn $e$ into a loop. By
the converse of Proposition~\ref{prop:typeD}, stated in
Remark~\ref{rem:types}, these are exactly the type $\mathscr{D}$ contracting
sets.

Let $\C$ be a type $\D$ contracting set and let
$f\in \D$ be a regular edge in the corresponding deleting set.
Let us call the pair $(\C,f)$ a {\em special pair} if $\C\cup\{e,f\}$
contains a cycle $C(\C,f)$ containing $e$ and $f$ has the smallest
label in $C(\C,f)\setminus\{e\}$. Observe that the cycle $C(\C,f)$
is unique since $\C\cup\{e\}$ contains no cycle
(equivalently, $\C$ is a contracting set of $G/e$ with respect to $\H$),
thus $\C\cup \{f\}$ contains at most one cycle. Furthermore $f$ is
externally active in $G/e$ exactly if it belongs to a special pair
$(\C,f)$.

Let us consider a special pair $(\C,f)$. As seen in \cite[Lemma
  3.7]{DHRel}, $\C^\p=\C\cup\{f\}$ is a contracting set
of $G$ that does not contain $e$. Thus $\C^\p$ is also a contracting set
of $G-e$ and, by the external activity of $f$ with respect to
$\C$ in $G/e$, the edge $f$ is the
element on the cycle $C(\C,f)$ with the smallest label, so it will be
internally active in $G- e$ but internally inactive in the
computation of $T_L(G,e)$ (since in the latter case $e$ is deleted
last). By the converse of Proposition~\ref{prop:typeC}, stated in
 Remark~\ref{rem:types}, $\C^{\prime}$ has type $\mathscr{C}$.

Conversely, let $\C^\p$ be a type $\mathscr{C}$ contracting set of $G$ with respect to
$\H^\prime$. By Proposition~\ref{prop:typeD}, the set
$\C^\p$ is a contracting set of $G-e$ with respect to $\H$ and $e$
closes a cycle (denoted by $C(\C^\p,e)$) with some edges from
$\C^\p$. Again, the cycle $C(\C^\p,e)$ is unique. Let $f$ be the
element in the set $\C^\p\cap C(\C^\p,e)$ with the smallest label.
We may use \cite[Lemma 3.7]{DHRel} again to see that
$\C^\p\setminus \{f\}\cup \{e\}$ is a contracting set of $G$ with
respect to $\H$ containing $e$ and it is easy to check that
$\C:=\C^\p\setminus \{f\}$ forms a special pair $(\C,f)$ with $f$.
We thus obtain a bijection between the type $\mathscr{C}$ contracting sets $\C^\p$ of
$G\setminus e$ and the special pairs $(\C,f)$ of $G$.

Let $\C$ be a type $\mathscr{D}$ contracting set of $G$ with respect to
$\H^\prime$ such that there is no edge $f$ in the corresponding deleting set
with the property that $(\C,f)$ is a special pair. By Proposition~\ref{prop:typeD}, the
set $\C$ is also a contracting set of $G/e$ with respect to $\H$. Also, contracting $e$
first or last does not affect the activities of the edges in $\C$ and $\D$.
If we contract all edges in $\C$ and
delete all edges in $\D$ first, $e$ will not be a loop in the resulting graph
$\Gamma(\C,e)$.  If $e$ is a bridge in $\Gamma(\C,e)$, then $\C$ makes
a contribution in $T_C(G,e)$ (after $e$ is contracted in
$\Gamma(\C,e)$). If $e$ is not a bridge in $\Gamma(\C,e)$, then it forms
a cycle with some zero edges in $\Gamma(\C,e)$. In this case $\C$ makes
a contribution to
$\pi_{/}T_{0}(G,e)$ after $e$ is contracted in $\Gamma(\C,e)$. So, in
the case that there are no special pairs $(\C,f)$, then either all edges
make the same contributions to $T(G/e)$ and $T_C(G,e)$ (if $e$ is a
bridge in $\Gamma(\C,e)$), or all edges make the same contributions to
$T(G/e)$ and $\pi_{/}T_{0}(G,e)$ (if $e$ is a not bridge in
$\Gamma(\C,e)$). Thus for all contracting sets $\C$ that do not form any
special pairs with edges from their corresponding deleting sets $\D$,
their total contribution to $T(G/e)-T_{0,/}(G,e)-
T_C(G,e)$ is zero.

Now assume that $\C$ is a type $\mathscr{D}$ contracting set of $G$ with respect to
$\H^\prime$ that forms  special pairs with some edges from $\D$. Without
loss of generality, assume that $f_1,\ldots,f_k\in \D$ are all the edges
that form special pairs with $\C$ and that they have been listed in the
increasing order according to their labels. Furthermore, let us assume
that the color of $f_i$ is $\mu_i$. Each $f_i$ is externally active in
$G/e$ hence their total contribution to $T(G/e)$ is $\prod_{i=1}^k
Y_{\mu_i}$. However, in the computation of $T_C(G,e)$ or $\pi_/T_{0}(G,e)$,
$e$ is to be contracted last (hence it has the smallest label) so the
total contribution of $f_1$, $f_2$, \ldots , $f_k$ to $T_C(G,e)$ or to
$T_{0,/}(G,e)$ is $\prod_{i=1}^k y_{\mu_i}$. All other edges have the
same contributions to both polynomials, since their activities are the
same, whether $e$ is contracted first or last. Thus the combined contribution of
$\C$ and $\D$ to the left hand side of (\ref{first}) is
$x_{\mu}\left(\prod_{i=1}^k Y_{\mu_i} -\prod_{i=1}^k y_{\mu_i}
\right)z_{[\Gamma(\C)]}$ times the product of the weights of the
edges that are different from $\{f_1,\ldots,f_k\}$, where $\Gamma(\C)$
is the graph of zero edges obtained after contracting the edges in
$\C\cup\{e\}$ and deleting the edges in the corresponding deleting set
$\D$. The term $(\prod_{i=1}^k
y_{\mu_i})z_{[\Gamma(\C)]}$ appears either in $T_C(G,e)$ or in
$\pi_{/}T_{0}(G,e)$, but it makes no difference in our argument. By
\cite[Lemma 3]{DHCl},
$$
x_{\mu}\left(\prod_{i=1}^k Y_{\mu_i}
-\prod_{i=1}^k y_{\mu_i} \right)=(Y_{\mu}-y_{\mu}) \sum_{i=1}^k x_{\mu_i}
\prod_{j=1}^{i-1} Y_{\mu_j} \prod_{j=i+1}^k
y_{\mu_j}.
$$
Thus it is sufficient to prove the following:
\begin{enumerate}
\item The graph $\Gamma(\C_i)-e$, obtained by contracting edges in
  $\C_i$ and deleting $e$ and the edges in $\D$, is (vertex pivot
  equivalent to) $\Gamma(\C,e)/e$.
\item The product of the weights of the edges $f_1,\ldots,f_k$ in the
contribution of $\C_i=\C\cup \{f_i\}$ to $T_L(G,e)$ is
$x_{\mu_i}\prod_{j=1}^{i-1}
Y_{\mu_j}\prod_{j=i+1}^k y_{\mu_j}$.
\item The weight of any edge $f\not\in\{e,f_1,\ldots,f_k\}$ is the same
in the contribution of $\C$ to $T(G/e)$ as in the contribution of
any $\C_i=\C\cup\{f_i\}$ to $T_L(G,e)$.
\end{enumerate}

The first statement above is true since $f_i$ and $e$ belong to the same cycle in which the other edges are all from $\C$. Thus deleting any one of the edges in the cycle and contracting the rest has the same effect. The vertices of all the edges involved become one single vertex.

Since $f_i\in\C_i$ and it belongs to $C(\C_i,e)$, it contributes a factor of
$x_{\mu_i}$ to $T_L(G,e)$ by the exception rule (since $e$ is deleted last, $f_i$ can never become a bridge). Since $C(\C_i,e)$ and $C(\C_j,e)$ both contain $e$, one can show that for any $i\not=j$, there exist a unique
cycle $C(\C_i,\C_j)$, consisting of $f_i$, $f_j$ and edges from
$C(\C_i,e)\cup C(\C_j,e)$, but not $e$ (see
\cite[Lemma 2]{DHCl}). Furthermore, if $i<j$, then the label of $f_i$ is
smaller than those of the other edges in this cycle. Since $f_j\in \D_i$
and $j>i$, $f_j$ has a larger label (recall that the labels of $f_1$,
$f_2$, \ldots , $f_k$ are in increasing order by our choice), it is
externally inactive. But if $j<i$, then $f_j$ has the smallest label
among the edges in $C(\C_i,\C_j)$ so it is externally active. Thus the
product of the weights of the edges $f_1,\ldots,f_k$ in the
contribution of $\C_i=\{f_i\}\cup \C\setminus\{e\}$ to $T_L(G,e)$ is
$x_{\mu_i}\prod_{j=1}^{i-1} Y_{\mu_j}\prod_{j=i+1}^k y_{\mu_j}$. This
proves the second statement above.

To prove the third statement, observe that a regular edge
$f\not\in\{e,f_1,\ldots,f_k\}$ either belongs to $\C$, in which case it would belong
to all contracting sets $\C_i$, or it belongs to $\D$, in which case it belongs to $\D_i$ for
each $i$, where $\D_i$ is the deleting set corresponding to $\C_i$.
Consider first the case $f\in \C$, i.e., $f\in \C_i$ for all $i$. For
each $i$ we have either $f\in C(\C_i,e)$ or $f\not\in C(\C_i,e)$. If
$f\in C(\C_i,e)$ then $f$ is internally inactive with
respect to $\C$ since $f$ has a label larger than that of $f_i$. It is
also internally inactive with respect to $\C_i$ in the computation of
$T_L(G,e)$ since $e$ is considered as an edge in the corresponding
deleting set and it has the smallest label (among all edges).
If $f\not\in C(\C_i,e)$ then its activity is the same with
respect to $\C$ or with respect to $\C_i$, since its activity is
determined by comparing its label with the labels of edges in the
deleting set that are on cycles containing $f$, yet $e$ and $f_i$ are
not on such cycles.  Thus a regular edge
$f\in\C\setminus\{e,f_1,\ldots,f_k\}$ has the same weight in the contribution of
$\C$ to $T(G/e)$ and in the contribution of $\C_i$ to $T_L(G,e)$.
In the second case, $f\in \D$, hence $f\in \D_i$ holds for all
$i$. Here $\D$, respectively $D_i$ is the deleting set corresponding to
$\C$, respectively $\C_i$. In this case $f$ either does not close any
cycle with edges from $\C\cup\{e\}$ that contains $e$, or it closes such a cycle
but it does not have the smallest label compared to other edges from
$\C$ on this cycle (since it is not one of
the $f_j$'s). If $f$ does not close any cycle with edges from
$\C\cup\{e\}$ that contains $e$, then $f$ will not close any cycle with
edges from $\C_i$ that also contains $f_i$, hence the determination of
its activity does not involve $e$ or  $f_i$, and it has the same
activity with respect to $\C$ and with respect to $\C_i$. Assume finally
$f$ closes a cycle $C(\C,f)$ with some edges from $\C\cup\{e\}$ and $e$
is on this cycle but $f$ does not have the smallest label among the
edges on this cycle. Then $f$ is externally inactive with respect to
$\C$. In this case $f$ also closes a (unique) cycle with edges from $\C_i$ that
contains $f_i$. Denote this cycle $C(\C_i,f)$. Let $g\in C(\C,f)$ be an
edge with label smaller than that of $f$. Then one can show that either
$g\in C(\C_i,f)$ or $g\in C(\C,f_i)$. If  $g\in C(\C_i,f)$, then $f$ is
externally inactive since $g$ has a smaller label. If $g\in C(\C,f_i)$,
then the label of $f_i$ is smaller than that of $f$ since $f_i$ has the
smallest label among the edges of $C(\C,f_i)$ (which contains $g$). So
$f$ is again externally inactive. To summarize, in all cases, $f$ has
the same weight in (the contribution of $\C$ to) $T(G/e)$ and in (the
contribution of $\C_i$ to) $T_L(G,e)$.

Equation \eqref{TLC2} is a direct generalization of
equation (6) in \cite[Theorem 2]{DHCl}. In \cite{DHCl} we invoked
matroid duality to derive this equation from the preceding one. We want
to avoid doing so this time since the presence of zero edges makes
questions of duality less clear, and since, in an effort to state our results in
a language that is more directly applicable in knot theory, we avoided
stating the matroid theoretic generalizations. Fortunately there is
another easy way to show \eqref{TLC2} after having shown \eqref{TLC1}:
it suffices to prove the validity of the {\em sum} of the two equations,
which is equivalent to
$$
x_{\mu}T_{/}(G,e) +y_\mu T_{-}(G,e)
=
X_\mu T_{C}(G,e)
+ Y_\mu T_{L}(G,e) .
$$
After adding
$x_{\mu}(T_{\H}^{\Lambda_0}(G/e)-T_{/}(G,e) )+
y_{\mu}(T_{\H}^{\Lambda_0}(G-e)-T_{-}(G,e) )$
to both sides, we obtain the equivalent equation
\begin{eqnarray}
&&x_{\mu} T_{\H}^{\Lambda_0}(G/e)+y_\mu T_{\H}^{\Lambda_0}(G-e)\nonumber\\
&=&
X_\mu T_{C}(G,e)
+x_{\mu}(T_{\H}^{\Lambda_0}(G/e)-T_{/}(G,e))\nonumber\\
&+& Y_\mu T_{L}(G,e) +
y_{\mu}(T_{\H}^{\Lambda_0}(G-e)-T_{-}(G,e) )\nonumber\\
&=&
X_\mu T_{C}(G,e)
+x_{\mu}\pi_{/}T_{0}(G,e)+Y_\mu T_{L}(G,e) +
y_{\mu}\pi_{-}T_{0}(G,e).\label{e48}
\end{eqnarray}
Let $G^\p$ be the colored graph that is identical to $G$, except that
the edge $e$ is colored with color $\mu$ instead of $\nu$. Let us
consider $T_{\H}^{\Lambda_0}(G^\p)$, whose definition is labeling
independent. We will show that both sides of (\ref{e48}) equal to
$T_{\H}^{\Lambda_0}(G^\p)$. First, by Corollary \ref{Cor2.7},
$T_{\H}^{\Lambda_0}(G^\p)=x_{\mu} T_{\H}^{\Lambda_0}(G/e)+y_\mu
T_{\H}^{\Lambda_0}(G-e)$ if we contract and delete $e$ first, since $e$
is neither a bridge nor a loop. That is, the left side of (\ref{e48}) is
equal to $T_{\H}^{\Lambda_0}(G^\p)$.
Next, let us now select any proper labeling such that the
label of $e$ is the smallest, so $e$ will be the last edge to be
contracted and/or deleted for each given contracting set $\C$. For each
type $\mathscr{C}$ contracting set $\C$ with respect to $\H^\p$, $e$ becomes a
loop after all edges of $\C$ have been contracted, hence it will
contribute a $Y_\mu$ term at the end. By Lemma \ref{l:TL}, the
collection of all such contracting sets are exactly those that make
non-zero contributions to $T_{L}(G,e) $, thus the
combined contributions of all such contracting sets yield $Y_\mu
T_{C}(G,e) $. Similarly, the combined contributions of
all type $\mathscr{D}$ contracting sets $\C$ of $G$ with respect to $\H^\p$ yield
exactly $X_\mu T_{L}(G,e) $ by Lemma
\ref{l:TC}. Finally, for each type zero contracting set $\C$, the edge $e$
becomes neither a bridge nor a loop, after all edges in
$\C$ have been contracted and all edges in $\D$ have been deleted.
In this case $\C$ contributes a term to $T_0(G,e)$. If, in the last
step, $e$ is contracted, we obtain a
term $x_\mu$ and $\C$ makes a contribution to $\pi_{/}T_{0}(G,e)$
by the definition of $\pi_/T_{0}(G,e)$. Similarly, if $e$ is deleted in
the last step, we get the expected $y_\mu$ term and $\C$ makes a
non-zero contribution to $\pi_{-}T_{0}(G,e)$. Combining the above, we
see that the right side of (\ref{e48}) is also equal to
$T_{\H}^{\Lambda_0}(G^\p)$, hence establishing the equality of
(\ref{e48}).
\end{proof}

In analogy to equations (8) and (9) in \cite{DHCl}, equations
\eqref{TLC1} and \eqref{TLC2} may be restated as
\begin{equation}\label{tctnmat1}
\det\left(\begin{array}{ll}
T_{L}(G,e) & T_{C}(G,e) \\
x_{\lambda} & y_{\lambda}
\end{array}\right)
= \det\left(\begin{array}{ll}
T_{L}(G,e) & T_{/}(G,e) \\
x_{\lambda}& Y_{\lambda}
\end{array}\right)
\end{equation}
and
\begin{equation}\label{tctnmat2}
\det\left(\begin{array}{ll}
T_{L}(G,e) ) & T_{C}(G,e) \\
x_{\lambda}& y_{\lambda}
\end{array}\right)
= \det\left(\begin{array}{ll}
T_{-}(G,e) &
T_{C}(G,e) \\
X_{\lambda}& y_{\lambda}
\end{array}\right).
\end{equation}
\begin{remark}{\em
The analogue of Theorem~\ref{DHRel} in \cite{DGH1} (and in \cite{DHCl})
is used to prove that the definition of the pointed Tutte polynomials
$T_C(G,e)$ and $T_L(G,e)$ is independent of the labeling, see \cite[Corollary
  2]{DHCl}. This time we do not prove labeling-independence of our
pointed relative Tutte polynomials, since it is obvious from the
definition. Note that this also applies to the special case when
$\H=\emptyset$. Thus the labeling independence of the polynomials
$T_C(G,e)$ and $T_L(G,e)$ defined in \cite{DGH1,DHCl} is
also a consequence of the labeling independence of the relative Tutte
polynomial shown in \cite{DHRel}.}
\end{remark}

\section{Contracting sets in a tensor product of graphs having zero edges}
\label{sec:cs}

A crucial idea behind proving the main results \cite[Theorem 5.1]{DGH1} and
\cite[Theorem 3]{DHCl}, providing a formula for the Tutte polynomial
of a tensor product $G_1\otimes_{\lambda} G_2$ of a colored connected graph
$G_1$ with a pointed colored connected graph $G_2$ (with distinguished
edge $e$) was to understand the composite  structure of a spanning tree
of $G_1\otimes_{\lambda} G_2$ in terms of considering an induced spanning tree
of $G_1$ and a collection of spanning trees of $G_2-e$ and $G_2/e$.
In this section we generalize this description to understanding the
composite structure of a contracting set and of the corresponding
deleting set in a tensor product
$G_1\otimes_{\lambda} G_2$ where both graphs may have zero edges.

Let $\Lambda$ be a color set, $\nu\not \in \Lambda$ is a distinguished color and $\Lambda^\p=\Lambda\cup\{\nu\}$.
From now on we assume that $G_1$
is a $\Lambda$-colored graph, together with a set of zero edges
$\H_1\subset E(G_1)$ which form a $\Lambda_0$-colored subgraph for some
$\Lambda_0\subset \Lambda$. Assume that the color
$\lambda\in \Lambda\setminus \Lambda_0$ appears in $E(G_1)$ as a
color of regular edges only. Let $G_2$ be a pointed $\Lambda^\p$-colored
graph with a distinguished edge $e$ that is neither a loop nor a
bridge, together with a set of zero edges $\H_2\subseteq E(G_2)\setminus
\{e\}$, which form a $\Lambda_0$-colored subgraph. To simplify our
arguments, we will assume that no edge of $G_2$ has color $\lambda$, $e$ is colored with $\nu$ and
that no other edges in $E(G_2)$ are colored with $\nu$.

As in \cite{DGH1} and \cite{DHCl}
we define the {\em $\lambda$-colored tensor product
$G_1\otimes_{\lambda} G_2$} as the graph obtained as follows. We associate a
distinct copy $G_2^f$ of $G_2$ to each edge $f$ of color $\lambda$ in $G_1$ by
identifying the edge $f$ with the copy $e_f$ of $e$ in $G_2^f$, and then
removing the identified edges $e_f$ and $f$. In particular, if $f$ is a
loop, then we will identify the endpoints of $e_f$ in $G_2^f$ (and
remove $e_f$). The resulting graph will contain the edges of
$\H_1$ and several copies of the edges of $\H_2$. We define the set
$\H$ of zero edges of $G_1\otimes_{\lambda} G_2$ as the set of all edges
belonging to $\H_1$ or any copy of $\H_2$.

Let us fix a contracting set $\C$ of $G_1\otimes_{\lambda} G_2$ with
respect to $\H$ and let $\D$ be the corresponding deleting set.
Let $f\in E(G_1)$ be of color $\lambda$ and let $G_2^f$ be
the copy of $G_2$ associated to $f$ with $e_f$ being the corresponding
distinguished edge of $G_2^f$. First we would like to make the following
fundamental observation on the intersection of $\C$ and $\D$ with
$E(G_2^f)$.
\begin{proposition}
\label{prop:local}
Let $\C_f=\C\cap E(G_2^f)$, $\D_f=\D\cap E(G_2^f)$ and
$\H_f=\H\cap E(G_2^f)$. Then
$\C_f$ is a contracting set of $G_2^f$ with respect
to $\H_f\cup \{e_f\}$ and $\D_f$ is the corresponding deleting
set.
\end{proposition}
\begin{proof}
Since $\C_f$ is a subset of $\C$, it clearly does not contain any
cycle. If $\D_f$ contains a cocycle, then after deleting the edges of
$\D_f$, there exist two vertices $v_1$ and $v_2$ that are not connected by a
path in $G_2^f-\D_f$. Since
$\D_f\subset\D$ and $\D$ does not contain any cocycle
in $G_1\otimes_\lambda G_2$, there must be a path $\gamma$ from $v_1$ to $v_2$ in
$G_1\otimes_\lambda G_2$. Since the only vertices where a path of $G_1\otimes_\lambda G_2$ can leave or
enter $G_2^f$ are the endpoints of $e_f$, the part of $\gamma$ that lies outside $G_2^f$
must form a path $\gamma'$ connecting the endpoints of $e_f$. Replacing
$\gamma'$ with $e_f$ results in a path in $G_2^f-\D_f$ connecting
$v_1$ and $v_2$, a contradiction to our assumption that $v_1$ and $v_2$
that are not connected by a path in $G_2^f-\D_f$.
\end{proof}

As a consequence of Proposition~\ref{prop:local} we may use
Definition~\ref{def:types} to classify the
pairs $(\C_f,\D_f)$ into type $\mathscr{C}$, type $\mathscr{D}$ and type zero. Using this
classification we define an {\em induced partition}
$(\C_1,\D_1,\widehat{\H_1})$ of $E(G_1)$ as follows:
\begin{itemize}
\item[(i)] $f\in \C_1$ if the color of $f$ is not $\lambda$ and
$f\in \C$, or the color of $f$ is $\lambda$ and $(\C_f,\D_f)$ has type $\mathscr{C}$;
\item[(ii)] $f\in \D_1$ if the color of $f$ is not $\lambda$ and
$f\in \D$, or the color of $f$ is $\lambda$ and $(\C_f,\D_f)$ has type $\mathscr{D}$;
\item[(iii)] $f\in \widehat{\H_1}$ if the color of $f$ is not
  $\lambda$ and  $f\in \H_1$, or the color of $f$ is $\lambda$ and
  $(\C_f,\D_f)$ has type zero.
\end{itemize}

\begin{proposition}
\label{prop:c1d1}
Let $\C$ be a contracting set of $G_1\otimes_{\lambda} G_2$ with
respect to $\H$ and let $\D$ be the corresponding deleting set. Let
$(\C_1,\D_1,\widehat{\H_1})$ be the induced partition of
$E(G_1)$. Then $\C_1$
is a contracting set of $G_1$ with respect to $\widehat{\H_1}$, and
$\D_1$ is the corresponding deleting set.
\end{proposition}
\begin{proof}
Assume, by way of contradiction, that $\C_1$ contains a cycle
$C=\{f_1,\ldots,f_k\}$. After replacing each $f_i$ of color $\lambda$
with a path in $\C$ connecting the endpoints of the distinguished edge
$e$ in the associated type $\mathscr{C}$ copy of $G_2$ we obtain a cycle in
$\C$, a contradiction. Therefore $\C_1$ cannot contain any cycle.
We obtain a similar contradiction if we assume that $\D_1$ contains a
cocycle $\{f_1,\ldots,f_k\}$, after replacing each $f_i$ of color
$\lambda$ with a minimal set of edges belonging to $\D$ in the
associated type $\mathscr{D}$ copy of $G_2$.
\end{proof}

We conclude this section with the converse of
Proposition~\ref{prop:c1d1}. Consider a colored graph $G_1$
and a pointed colored graph $G_2$ subject to the assumptions made at the
beginning of this section.
Let $\H_\lambda$ be a subset of the $\lambda$ colored edges of $G_1$ and let $\widehat{\H_1}=\H_1\cup\H_\lambda$. Let $\C_1$ be a contracting set of $G_1$
with respect to $\widehat{\H_1}$ and let $\D_1$ be the corresponding
deleting set.  For each
edge $f\in E(G_1)$ of color $\lambda$, let $G_2^f$ be the copy of $G_2$
associated to $f$ in $G_1\otimes_\lambda G_2$ with $e_f$ being the
distinguished edge. Let us select a contracting set $\C_f$ (and the
corresponding deleting set $\D_f$) of $G_2^f$
respect to $\H_f\cup\{e_f\}$ in the following way: if $f\in \C_1$,
we select a pair $(\C_f,\D_f)$ of type $\mathscr{C}$, if $f\in \D_1$, we select a
pair $(\C_f,\D_f)$ of type $\mathscr{D}$ and if $f\in \H_\lambda$, we select a
pair $(\C_f,\D_f)$ of type zero. Let $\C$ be the
union of all edges in $\C_1$ whose color is not $\lambda$ and of all
the sets $\C_f$ and let $\D$ be the union of all edges in $\D_1$ not
whose color is not $\lambda$ and of all the sets $\D_f$.

\begin{theorem}
\label{thm:bigCD}
The edge set $\C$ defined above is a
contracting set with respect to $\H$ and $\D$ is the corresponding
deleting set.
\end{theorem}
\begin{proof}
Clearly, $\D=E(G_1\otimes_{\lambda} G_2)\setminus
(\C\cup \H)$, so we only need to prove that $\C$ contains no cycle and
$\D$ contains no cocycle.

Assume, by way of contradiction, that $\C$ contains a cycle $C$. This cycle
cannot be contained entirely in a copy $G_2^f$ of $G_2$ since no set
$\C_f$ contains a cycle. Thus $\C_f$ must
be either the empty set or a path $\gamma_f$ connecting the endpoints of
$e_f$. In the latter case
$f\in E(G_1)$ must belong to $\C_1$ since $\gamma_f\cup\{e_f\}$ forms a cycle hence  is of type $\mathscr{C}$. After replacing each such
path $\gamma_f$ with the edge $f\in \C_1$ we obtain a cycle contained in
$\C_1$, in contradiction with $\C_1$ being a contracting set.

To show that $\D$ contains no cocycle it suffices to show that for every
edge $g\in \D$ there is a walk contained in $\C\cup\H$ connecting the
endpoints of $g$. We will have two cases, depending on whether $g$
belongs to $E(G_1)$ or it belongs to a copy $G_2^f$ of $G_2$.
Consider first the case $g\in E(G_1)$. Then $g\in\D_1$, and there is a
path $\gamma_1(g)$ contained in $\C_1\cup\widehat{\H_1}$ connecting the endpoints of
$g$ in $G_1$. If the color of an edge $h$ in $\gamma_1(g)$ is not $\lambda$
then this edge also belongs to $E(G_1\otimes_{\lambda}G_2)$.
If the color of $h\in \gamma_1(g)$ is $\lambda$ then $(\C_h,\D_h)$ has type $\mathscr{C}$ or zero and we may replace $h$ with a path $\gamma(h)$ connecting
the endpoints of $h$ in $\C\cup \H$. Thus we obtain a walk in $\C\cup\H$
that connects the endpoints of $g$.
Consider finally the case when $g$ belongs to a copy $G_2^f$ of
$G_2$. Then $g\in \D_f$, and there is a path $\gamma_f(g)$ in
$\C_f\cup\H_f\cup\{e_f\}$ connecting the endpoints of $g$. If
$\gamma_f(g)$ does not contain $e_f$ then all of its edges belong to
$\C\cup\H$ and we are done. Thus we may assume $\gamma_f(g)$ contains
$e_f$. If $f$ belongs to $\C_1\cup\widehat{\H_1}$ then, in analogy to the
previous case, we may replace $e_f$ with a path
$\gamma_f(e_f)$ connecting its endpoints in $\C_f\cup\H_f$ and obtain a
walk contained in $\C\cup \H$ connecting the endpoints of $g$. We
are left with the case when $f$ belongs to $\D_1$.
Repeating the argument of the case $g\in E(G_1)$ for $f$, there
is a walk connecting the endpoints of $f$ in $\C_1\cup
\D_1$ which may be transformed into a path $\gamma(f)$ connecting the
endpoints of $f$ in $\C\cup \H$. We may replace $e_f$ in $\gamma_f(g)$
with $\gamma(f)$  and obtain a walk connecting the endpoints of $g$ in
$\C\cup\H$.
\end{proof}

\section{The tensor product formula}
\label{sec:tp}

This section contains the main result of our paper. The issue here is to
find a way to compute the relative Tutte polynomial of
$G_1\otimes_\lambda G_2$ in terms of the relative Tutte polynomial of
$G_1$ and the pointed relative Tutte polynomials of $G_2$ via some
suitable variable substitutions. In the case that there are no zero
edges involved, this is done by keeping
all variables of color $\mu\neq \lambda$ in $T(G_1)$ unchanged,
and using the substitutions
$
X_{\lambda}\mapsto T(G_1- e)$, $x_{\lambda}\mapsto T_L(G_1,e)$,
$Y_{\lambda}\mapsto T(G_1/e)$ and $y_{\lambda}\mapsto T_C(G_1,e)
$, see \cite[Theorem 5.1]{DGH1} and \cite[Theorem 3]{DHCl}. The new
obstacle we face here (when there
are zero edges present) is
that a choice of contracting set in $G_1\otimes_\lambda G_2$ may turn
some $G_2$ copies into the zero types by the results we have established
in the previous sections. That is, a $\lambda$ colored edge in $G_1$ may
not always be treated as a regular edge (which is either in a
contracting set or a deleting set). To simplify our arguments,
in this section we also assume that $G_1$ is a connected graph and that
$G_2$ is a connected graph in which the pointed edge $e$ is neither a
loop nor a coloop. To distinguish an edge of color $\lambda$, treated
as zero edge, from the regular edge of the same color, we will change
its color to a new color $\lambda_0\not\in \Lambda$, as defined  more
formally in the following definition.

\begin{definition}
Let $G$ be any $\Lambda$-colored graph together with a set of zero edges
$\H$, let $\lambda\in \Lambda$ be a color used to color regular edges
only and let $S$ be any subset of the set $E_{\lambda}\subseteq
E(G)\setminus \H$ of $\lambda$-colored edges. Let $\lambda_0\not\in
\Lambda$ be a new color. We define the graph $G_S$ as the
$\Lambda\cup\{\lambda_0\}$-colored graph obtained from $G$ by changing
the color of each edge belonging to $S$ to $\lambda_0$.
\end{definition}

In analogy to \cite[Theorem 5.1]{DGH1} and
\cite[Theorem 3]{DHCl}, we will express
$T_\H^{\Lambda_0}(G_1\otimes_\lambda G_2)$ as a
function of graph polynomials associated to $G_2$ and all the graphs of the form
$G_{1S}$, respectively. Since we are dealing with more than just the graphs $G_1$ and $G_2$ (as is the case when there are no zero edges involved), the  procedure is much more complex, we will break down our process into a sequence of homomorphisms and
$\T(\R,\Lambda\setminus \Lambda_0)$-linear maps. The first homomorphism
applied is a direct generalization of the substitutions used in
\cite[Theorem 5.1]{DGH1} and \cite[Theorem 3]{DHCl}, which generalize a
transformation introduced by Brylawski~\cite{B,BO}.

\begin{definition}
Let $G$ be a pointed $\Lambda^\p$-colored graph,
together with a $\Lambda_0$-colored subgraph $\H$ of zero edges (recall that $\Lambda^\p=\Lambda\cup\{\nu\}$ and $\nu$ is the unique color for the pointed edge $e$). We define the {\em regular Brylawski map $\beta_{\lambda,G}\in
  \operatorname{End}(\T(\R,\Lambda,\Lambda_0))$}, as the endomorphism
sending each variable $x_\mu$, $X_\mu$, $y_\mu$,
$Y_\mu$ such that
$\mu\neq \lambda$ into itself and sending
$X_{\lambda}$ into $T_-(G,e)$, $x_{\lambda}$ into
$T_{L}(G,e)$,
$Y_{\lambda}$ into $T_{/}(G,e)$ and $y_{\lambda}$ into
$T_{C}(G,e)$. We will use $\beta_\lambda$ for $\beta_{\lambda,G}$ when the graph $G$ is clear in the context of the problem.
\end{definition}

In analogy to \cite[Lemma 4]{DHCl}, the fact that the regular Brylawski
map is well-defined is a direct consequence of equations
\eqref{tctnmat1} and \eqref{tctnmat2}. Furthermore, for each variable $z_{[\Gamma]}$
that appears in the polynomials $T_{-}(G,e)$,
$T_{L}(G,e)$, $T_{/}(G,e)$ and
$T_{C}(G,e)$, the corresponding graph $\Gamma$ has
no edge of color $\nu$, thus $\beta_\lambda$ indeed takes
the ring $\T(\R,\Lambda,\Lambda_0)$ into itself. However, when applying $\beta_\lambda$ to a relative Tutte polynomial
of a graph with $k\ge 2$ $\lambda$ colored edges, the result is a polynomial containing terms of the form
$z_{[\Gamma_1]}\cdot z_{[\Gamma_2]}\cdots z_{[\Gamma_k]}$, which is not a universal relative Tutte polynomial (recall that a universal relative Tutte polynomial is a linear combination of the terms of the form $z_{[\Gamma]}$). The need to change this into a universal relative Tutte polynomial leads to the next definition.

\begin{definition}
The {\em regular splicing map} $\sigma:
  \T(\R,\Lambda,\Lambda_0)\rightarrow \T(\R,\Lambda,\Lambda_0)$ is the
$\T(\R,\Lambda\setminus\Lambda_0)$-linear map
induced by $\sigma(z_{[\Gamma_1]}\cdots
z_{[\Gamma_k]})
=z_{[\Sigma([\Gamma_1],\ldots, [\Gamma_k])]}$. Here
  $\Sigma([\Gamma_1],\ldots, [\Gamma_k])$ is the connected graph obtained
  by the repeated splicing of all connected components of the graph
  $\Gamma_1\uplus\cdots \uplus \Gamma_k$ (here $\uplus$ stands for disjoint union).
\end{definition}

When we perform the vertex splicing operation on a pair of connected
graphs, the resulting graph is connected and unique up to vertex pivot
equivalence. It follows by induction on the number of connected
components $m$ of $\Gamma_1\uplus\cdots \uplus \Gamma_k$ that, repeating
the vertex splicing operation $m-1$ times in such a way that each
splicing operation merges vertices from different connected components,
results in a connected graph which is unique up to vertex pivot
equivalence.  Thus  $\sigma$
is well defined.

For each $S\subseteq E_{\lambda}(G_1)$, the polynomial
$\sigma\beta_\lambda
(T^{\Lambda_0\cup\{\lambda_0\}}_{\H_1\cup
  S}(G_{1S}))$ (where $\beta_\lambda=\beta_{\lambda,G_2}$) is a $\T(\R,\Lambda\setminus\Lambda_0)$-linear
combination of the variables $\{z_{[\Gamma]}\::\: [\Gamma]\in
\VPE(\Lambda_0\cup{\lambda_0})\}$. In other words,
$$
\sigma\beta_{\lambda}
(T^{\Lambda_0\cup\{\lambda_0\}}_{\H_1\cup
  S}(G_{1S}))\in \bigoplus_{[\Gamma]\in
\VPE(\Lambda_0\cup{\lambda_0})} \T(\R,\Lambda\setminus\Lambda_{0})
z_{[\Gamma]}.$$
We will use this module as the domain of the {\em zero Brylawski map},
to be defined below. This map is analogous to the maps introduced  by
Brylawski~\cite{B,BO} (and generalized in \cite[Theorem 5.1]{DGH1} and
\cite[Theorem 3]{DHCl}) only in the sense that they are associated to
replacing edges in (a recolored variant of) $G_1$ with copies of $G_2$. Let $G_2$ be a graph with a distinguished (pointed) edge colored with the unique color $\nu$. Recall that $T_{0}(G_2,e)$ is a $\T(\R,\Lambda\setminus\Lambda_0)$-linear
combination of terms of the form $z_{[\Gamma]}$ where each $\Gamma$ contains exactly one edge of color $\nu$. Thus we have
$
T_0(G_2,e)=\sum_{1\le j\le n}p_j\cdot z_{[\Gamma_j]}
$
where $p_j\in  \T(\R,\Lambda\setminus\Lambda_{0})$ and $\Gamma_j$
contains exactly one $\nu$-colored edge for each $j$. Let
$\textbf{P}=\{p_1,p_2,\ldots ,p_n\}$. For any graph $\Gamma$ consisting
of only zero edges and $k$ $\lambda_0$-colored edges, let us number its
$\lambda_0$-colored edges by $1$, $2$, \ldots , $k$ in an arbitrary
way. For any choice of $q_1=p_{j_1}\in  \textbf{P}$, $q_2=p_{j_2}\in
\textbf{P}$, \ldots , $q_k=p_{j_k}\in \textbf{P}$, let
$\Gamma_{q_1,\ldots ,q_k}$ be the graph obtained by identifying the
$i$-th $\lambda_0$-colored edge in $\Gamma$ with the $\nu$-colored edge
in $\Gamma_{j_i}$ first, then removing the identified edge, for each
$1\le i\le k$ (so $\Gamma_{q_1,\ldots,q_k}$ is obtained through a total of $k$ 2-sum operations on $\Gamma$).

\begin{definition}
We define the {\em zero Brylawski map}
$\beta_{0,G_2}$ as the
$\T(\R,\Lambda\setminus\Lambda_{0})$-linear map from
$\bigoplus_{[\Gamma]\in
\VPE(\Lambda_0\cup{\lambda_0})} \T(\R,\Lambda\setminus\Lambda_{0})
z_{[\Gamma]}$
to
$ \T(\R,\Lambda,\Lambda_{0})$
induced by the mapping that sends each
$z_{[\Gamma]}$ to the symmetric sum
$$
\sum_{q_1\in \textbf{P},\ldots ,q_k\in \textbf{P}}q_1q_2\cdots q_k\cdot
z_{[\Gamma_{q_1,\ldots ,q_k}]}.
$$
\end{definition}

Notice that although this definition presupposes labeling the $\lambda_0$-colored edges of
$\Gamma$ in some order, the end result is independent of the choice
of this labeling, since the summation is symmetric hence is invariant under the permutations of the factors. Again we will use the short hand notation $\beta_0$ for $\beta_{0,G_2}$ when $G_2$ is clear from the context of the problem.

\begin{theorem}\label{TensorFormula}
The universal $\Lambda_0$-colored relative Tutte polynomial
$T_\H^{\lambda_0}(G_1\otimes_\lambda G_2)$ is given by
$$
T_\H^{\lambda_0}(G_1\otimes_\lambda G_2)=
\sum_{S\subseteq E_{\lambda}(G_1)}
\Phi(T^{\Lambda_0\cup\{\lambda_0\}}_{\H_1\cup
  S}(G_{1S})),
$$
where $\Phi=
\beta_{0}\circ
\sigma\circ\beta_{\lambda}$.
\end{theorem}
\begin{proof}
We generalize the proofs of \cite[Theorem 5.1]{DGH1} and
\cite[Theorem 3]{DHCl}, using the description of the contracting and
deleting sets of $G_1\otimes_\lambda G_2$ given in Section~\ref{sec:cs}.
The polynomial $T_\H^{\lambda_0}(G_1\otimes_\lambda G_2)$ is the total
weight of all contracting sets $\C$ of the graph $G_1\otimes_\lambda
G_2$, calculated using any proper labeling of the edges of $G_1\otimes_\lambda
G_2$ with respect to $\H$. Let us select our proper labeling in three
steps as follows:
\begin{itemize}
\item[(i)] Label all edges in $\H$ with zero.
\item[(ii)] Label all regular edges (including the $\lambda$-colored edges) of
  $G_1$ with pairwise distinct positive integers that are multiples of
  $|E(G_2)|$.
\item[(iii)] If the label of a $\lambda$-colored edge $f\in E_{\lambda}(G_1)$
  is $k\cdot |E(G_2)|$ for
  some $k>0$ then number the regular edges in the copy of $G_2$
  replacing $f$ using the elements of the set $\{(k-1)|E(G_2)|+1,
  (k-1)|E(G_2)|+2,\ldots, k|E(G_2)|\}$.
\end{itemize}
We obtain a proper labeling of $G_1\otimes_\lambda G_2$ with respect to $\H$
that has the following property: if we list the regular edges of of
$G_1\otimes_\lambda G_2$ in increasing order of labels, the regular
edges belonging to a copy of $G_2$ associated to the same
$f\in E_{\lambda}(G_1)$ form a sublist of consecutive elements.
By the description given in Section~\ref{sec:cs}, there is a one to one correspondence between the contracting sets $\C$ of $G_1\otimes_\lambda G_2$ and the contracting sets generated by the following three-step
procedure:
\begin{itemize}
\item[(1a)] Select a subset $S$ of $E_{\lambda}(G_1)$ and define
  $\widehat{\H_1}:=\H_1\cup S$.
\item[(1b)] Select a contracting/deleting set pair $\C_1$, $\D_1$ of $G_1$ with respect to
  $\widehat{\H_1}$.
\item[(2)] For each copy $G_2^f$ of $G_2$, associated to an edge $f\in
  E_{\lambda} (G_1)$, partition the set of regular edges
  $E(G_2^f)\setminus (\H_f\cup \{f\})$ of the copy into a contracting set
  $\C_f$ and a deleting set $\D_f$ such that $(\C_f,\D_f)$ has type $\mathscr{C}$
  (or type $\mathscr{D}$, or type zero, respectively) exactly when $f\in\C_1$
  (or $f\in \D_1$, or $f\in \widehat{\H_1}$, respectively).
\end{itemize}
We then define the  contracting set $\C$ as the union of the sets $\C_f$ and
of $\C_1\setminus E_{\lambda}(G_1)$. The corresponding deleting set is the
union of the sets $\D_f$ and of $\D_1\setminus E_{\lambda}(G_1)$. Note
that there is no other restriction on the choices made in the three
steps of the above procedure than the ones stated. If we group the
weights of the contracting sets $\C$ according to the choices in the
above procedure,  the choice made in step (1a) corresponds to summing over
all subsets $S$ of $E_{\lambda}(G_1)$. After fixing $S$, summing over
all possible choices $\C_1$ in step (1b) calls for summing over the same
contracting sets that are used to compute the relative
Tutte polynomial $T^{\Lambda_0\cup\{\lambda_0\}}_{\H_1\cup
  S}(G_{1S})$. For a fixed contracting set $\C_1$ of
$G_{1S}$ with respect to $\widehat{\H_1}=\H_1\cup S$, step (2)
calls for summing over all contracting sets $\C_f$ of the same type,
where the type depends on $f$ belonging to $\C_1$, $\D_1$ or
$\widehat{\H_1}$. A key observation in understanding the rest of the proof below
is that we may replace steps (1a) and (1b) above with the following
step.
\begin{itemize}
\item[(1)] In decreasing order of their labels, put each $f\in
  E_{\lambda}(G_1)$ into $\C_1$, $\D_1$ or $S$, subject to the following
  restrictions: an edge $f\in E_{\lambda}(G_1)$ cannot be put into
  $\C_1$ if it becomes a loop in $G_1$ after contracting all higher
  labeled edges of $\C_1$, and it cannot be put into $\D_1$ if it
  becomes a bridge in $G_1$ after deleting all higher labeled edges of $\D_1$.
\end{itemize}
The restrictions are necessary and sufficient to guarantee that the
resulting $\C_1$ is a contracting set and the resulting $\D_1$ is the
corresponding deleting set with respect to $\widehat{\H_1}=\H_1\cup S$.
Furthermore, as noted in the alternative
Definition~\ref{acti_alt_def}, the external or internal activity of an
edge in $\C_1$ or $\D_1$ is determined by whether the edge in
question becomes a loop or bridge after contracting all higher labeled
edges of $\C_1$ and deleting all higher labeled edges of $\D_1$.
Since the labeling on the edges of $G_1\otimes_{\lambda} G_2$ is
obtained by replacing each $\lambda$-colored edge $f$ by a consecutive run
of edges of $G_2^f$, and due to the special dependence of the pair $(\C,\D)$
on the pair $(\C_1,\D_1)$, the above observation regarding activities may be
extended to $G_1\otimes_{\lambda} G_2$ in the following way. Consider
any $f\in E_\lambda(G_1)$. Contract all edges of $\C$ and delete all
edges of $\D$ whose label is higher than the label of any edge in the
copy $G_2^f$ of $G_2$. After performing these operations, the endpoints
of $f$ get identified if and only if $f$  becomes a loop after
contracting all higher labeled edges of $\C_1$ in $G_1$. Similarly, the endpoints of $f$ are only
connected by paths in the copy $G_2^f$  if and
only if $f$ becomes a bridge after deleting all higher labeled edges of
$\D_1$ in $G_1$. The first observation is true because contracting any $\lambda$-colored edge
 $g\in\C_1$ in $G_1$ (with higher label) identifies the endpoints of $g$
in $G_1$ and the same effect is achieved by contracting all edges in $\C_g$ in
$G_1\otimes_\lambda G_2$ as  $(\C_g,\D_g)$ has type $\mathscr{C}$.
Similarly, deleting any $\lambda$-colored edge
$g\in\D_1$ in $G_1$ (with higher label) removes the possibility of going from one endpoint
of $g$ to the other endpoint, without visiting any other vertex of
$G_1$, and the same effect is achieved by deleting all edges in $\D_g$ in
$G_1\otimes_\lambda G_2$, as $(\C_g,\D_g)$ has type $\mathscr{D}$.

We now consider three cases depending on whether an edge $f\in
E_{\lambda}(G_1)$ becomes a loop, a bridge, or neither after
contracting all higher labeled edges of $\C_1$ and deleting all higher
labeled edges of $\D_1$ in step (1) above.

{\bf\noindent Case 1}. $f\in E_{\lambda}(G_1)$ becomes neither a loop nor
a bridge after contracting all higher labeled edges of $\C_1$ and
deleting all higher labeled edges of $\D_1$. In this case, $f$ may be
put into either of $\C_1$, $\D_1$, or $S$. Furthermore, if we use
the contraction/deletion formula (\ref{recur1}) (in the decreasing order of
the labels of the edges) to compute $T(G_1\otimes_\lambda G_2)$ then,
after contracting all edges in $\C$ and deleting all edges in $\D$ whose
label is higher than the label of the edges in $G_2^f$, the endpoints of
$f$ are still distinct, and there is a path outside the copy $G_2^f$ of
$G_2$ connecting the endpoints of $f$.

{\em Subcase 1(a).} $f\in \C_1$. Since $f$ is inactive, it contributes a
term $x_\lambda$ to $T^{\Lambda_0\cup\{\lambda_0\}}_{\H_1\cup  S}(G_{1S})$.
The total weight of all type $\mathscr{C}$ contracting sets of $G_2^f$
is precisely $T_L(G_2,e)$ (Lemma \ref{l:TL}).

{\em Subcase 1(b).} $f\in \D_1$. Since $f$ is inactive, it contributes a term $y_\lambda$ to $T^{\Lambda_0\cup\{\lambda_0\}}_{\H_1\cup
  S}(G_{1S})$. The total weight of all type $\mathscr{D}$ contracting
sets of $G_2^f$ is precisely $T_C(G_2,e)$ (Lemma \ref{l:TC}).

{\em Subcase 1(c).} $f\in S$.  and is ``inactive" in the sense that if we
follow the contraction/deletion formula (in the decreasing order of the
labels of the edges) The total weight of of all type zero contracting
sets of $G_2^f$ is precisely $T_0(G_2,e)$ (Lemma \ref{l:T0}).

Notice that in all three cases above, replacing $f$ by any terminal
graph resulted from a contracting set of the corresponding type (through
the splicing or the 2-sum operations  defined in the mappings
$\beta_\lambda$, $\sigma$ and $\beta_{0}$) will not affect the
activities of the remaining edges. Thus, 1(a) and 1(b) prove the
validity of the substitutions $x_\lambda\to T_L(G_2,e)$ and
$y_\lambda\to T_C(G_2,e)$, while 1(c) shows the validity of replacing a
$\lambda_0$ colored edge in the terminal graph of  $G_{1S}$ by a copy of
$T_0(G_2,e)$.

{\bf\noindent Case 2}. $f\in E_{\lambda}(G_1)$ becomes
a bridge after contracting all higher labeled edges of $\C_1$ and
deleting all higher labeled edges of $\D_1$. In this case, $f$ may be
put into either of $\C_1$ or $S$. Furthermore, if we use
the contraction/deletion formula (\ref{recur1}) (in the decreasing order of
the labels of the edges) to compute $T(G_1\otimes_\lambda G_2)$ then,
after contracting all edges in $\C$ and deleting all edges in $\D$ whose
label is higher than the label of the edges in $G_2^f$, the endpoints of
$f$ are still distinct, but there is no path outside the copy $G_2^f$ of
$G_2$ connecting the endpoints of $f$. If we choose to put $f$ into
$\C_1$, it becomes internally active in $G_{1S}$ thus contributing a
term $X_\lambda$ to $T^{\Lambda_0\cup\{\lambda_0\}}_{\H_1\cup  S}(G_{1S})$.

In this case, the combined
contribution of all type $\mathscr{C}$ and type zero contracting sets of
$G_2^f$ is
$T_\H^{\Lambda_0}(G_2^f-e_f)=T_\H^{\Lambda_0}(G_2^f-e_f)-\pi_-
T_0(G_2^f,e_f)+\pi_- T_0(G_2^f,e_f)=T_-(G_2,e)+\pi_- T_0(G_2,e)$. Since
we need to replace $X_\lambda$ by a polynomial that is label
independent, we will simply substitute $T_-(G_2,e)$ into $X_\lambda$. On
the other hand, if $f\in S$, then we will still replace the
corresponding $\lambda_0$ colored edge in the terminal graph of $G_{1S}$
by a copy of $T_0(G_2^f,e_f)$. In this particular case, the terminal
graphs obtained by using the type zero contracting sets in $\pi_-
T_0(G_2^f,e_f)$ through splicing and using the corresponding type zero
contracting sets in $T_0(G_2^f,e_f)$ through 2-sum operation are in fact
vertex pivot equivalent (since the end points of $f$ are cut vertices)
and their total contributions are the same. Thus the combined
contribution of  $T_-(G_2,e)$ (for $f\in \C_1$) and $T_0(G_2,0)$ (for
$f\in S$) is equal to $T_\H^{\Lambda_0}(G_2^f-e_f)$, which is the
correct contribution of $G_2^f$ in this case. The key point of this
substitution rule is that the $\lambda_0$ colored edges are now treated
equally in terms of substitution and we have found a right substitution
for $X_\lambda$ that is label independent.

{\bf\noindent Case 3}. $f\in E_{\lambda}(G_1)$ becomes
a loop after contracting all higher labeled edges of $\C_1$ and
deleting all higher labeled edges of $\D_1$. In this case, $f$ may be
put into either of $\D_1$ or $S$. Furthermore, if we use
the contraction/deletion formula (\ref{recur1}) (in the decreasing order of
the labels of the edges) to compute $T(G_1\otimes_\lambda G_2)$ then,
after contracting all edges in $\C$ and deleting all edges in $\D$ whose
label is higher than the label of the edges in $G_2^f$, the endpoints of
$f$ become identical. If we choose to put $f$ into
$\C_1$, it becomes externally active in $G_{1S}$ thus contributing a
term $Y_\lambda$ to $T^{\Lambda_0\cup\{\lambda_0\}}_{\H_1\cup
  S}(G_{1S})$.

In this case the combined contribution of all type $\mathscr{D}$ and
type zero contracting sets of $G_2^f$ is
$T_\H^{\Lambda_0}(G_2^f/e_f)=T_\H^{\Lambda_0}(G_2^f/e_f)-\pi_/T_0(G_2^f,e_f)+\pi_/
T_0(G_2^f,e_f)=T_/(G_2,e)+\pi_/T_0(G_2,e)$. As we did in 2(a), we will
simply substitute $Y_\lambda$ by $T_/(G_2,e)$. On the other hand, if
$f\in S$, then we will still substitute the corresponding $\lambda_0$
colored edge in the terminal graph of $G_{1S}$ by a copy of
$T_0(G_2^f,e_f)$. Again, the terminal graphs obtained by using the type
zero contracting sets in $\pi_/T_0(G_2^f,e_f)$ through splicing and
using the corresponding type zero contracting sets in $T_0(G_2^f,e_f)$
through 2-sum operation are also vertex pivot equivalent (the end points
of $f$ are identified and is also a cut vertex in $G_1\otimes_\lambda
G_2$) and their total contributions are the same. Thus this substitution
rule is also valid.
\end{proof}

\section{Examples and ending remarks}\label{conc}

Let us end this paper by a couple of examples and remarks.

\begin{remark}{\em
In the case that $G_1$ contains zero edges but $G_2$ does not, there are no type zero contracting sets in $G_2$ and we have $T_-(G_2,e)=T(G_2-e)$,
$T_/(G_2,e)=T(G_2/e)$. In this case the substitution rule obtained in this paper is the same as the one given \cite{DHCl}. That is, the main result in \cite{DHCl} can be extended to the relative Tutte polynomial of $G_1\otimes_\lambda G_2$ without having to modifying the definitions of the pointed Tutte polynomials of $G_2$ and the substitution formula.}
\end{remark}

\begin{remark}{\em
Another extreme (and trivial) example is when $G_2$ consists of only two edges which are not loop edges: one is the special edge $e$ and the other a zero edge. In this case any graph $G_1$ with zero edges can be obtained by color the zero edges by $\lambda$ and then take the tensor product  $G_1\otimes_\lambda G_2$. In this case all pointed Tutte polynomials are zero except $T_0(G_2,e)$. Consequently, the only non-trivial substitution (as expected) happens only when the set $S$ contains all $\lambda$-colored edges. }
\end{remark}

Next, let us use a relatively simple example to illustrate the application of Theorem \ref{TensorFormula}.
Figure~\ref{G1G2} shows the graphs $G_1$ and $G_2$, as well as their
corresponding tensor product $G_1\otimes_\lambda G_2$. We will assume
that the regular edges in $G_2$ are labeled in such a way that the top
edge has the highest label and the bottom edge has the lowest label, see
the numbers $1$ through $3$ in Figure~\ref{G1G2}.

\vskip -.1in
\begin{figure}[htbp]
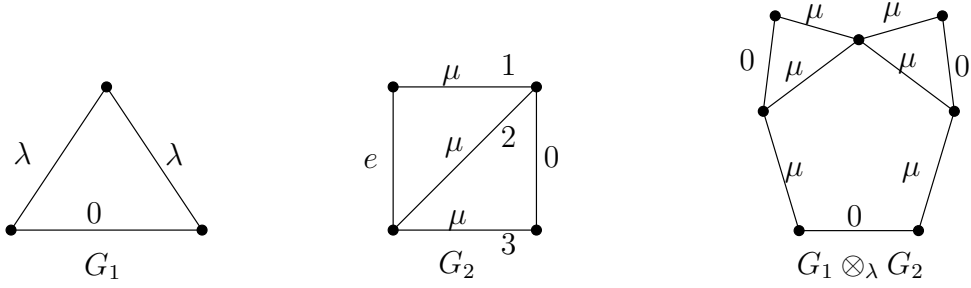

\begin{center}
\input{g1g2.pspdftex}\hskip 1.0in \input{g1timesg2.pspdftex}
\caption{The graphs $G_1$, $G_2$ and $G_1\otimes_\lambda G_2$. }\label{G1G2}
\end{center}
\end{figure}

For the two regular edges in $G_1$ that are $\lambda$-colored, there are three cases: none of them is in $\widehat{\H_1}$, one of them is in $\widehat{\H_1}$ (and there are two symmetric cases here) and both are in $\widehat{\H_1}$. Using Corollary \ref{Cor2.7}, we get (the details are left to our reader)
\begin{equation}\label{TG1}
\sum_{S\subseteq E_{\lambda}(G_1)}
T^{\Lambda_0\cup\{\lambda_0\}}_{\H_1\cup
  S}(G_{1S})=x_\lambda^2 z{[\includegraphics[scale=.25]{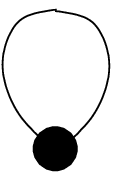}]}+y_\lambda(x_\lambda + X_\lambda)z{[\includegraphics[scale=.25]{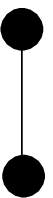}]}
  +2x_\lambda z{[\includegraphics[scale=.25]{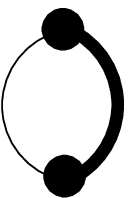}]}+2y_\lambda z{[\includegraphics[scale=.2]{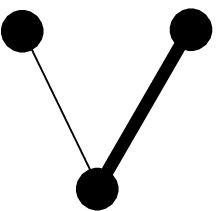}]}+z{[\includegraphics[scale=.2]{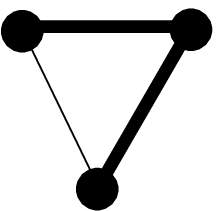}]},
\end{equation}
where the thicker edges in the graphs are of color $\lambda_0$ and the rest are zero edges. Next we will compute the pointed relative Tutte polynomials
$T_C(G_2,e)$, $T_L(G_2,e)$, $T_0(G_2,e)$, $T_/(G_2,e)$ and
$T_-(G_2,e)$. To help our reader, we list all possible contracting sets and their contributions in Table~\ref{tab:contrib}. In the case of a type zero contracting set, the edge $e$ in the terminal graph is marked by a thickened and dashed line.
\begin{table}[h]
\begin{center}
\begin{tabular}{|c|c|c|}
\hline
$\C$ & Type of $\C$ & Contributions\\
\hline
$\{1\}$ & $\mathscr{D}$ &
$X_{\mu} y_{\mu}^2z{[\includegraphics[scale=.25]{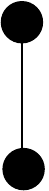}]}$ to $T_C$, $T(G_2/e)$\\
$\{2\}$ & $\mathscr{D}$ &$x_{\mu} y_{\mu}^2z{[\includegraphics[scale=.25]{Gamma12}]}$ to $T_C$, $x_{\mu} y_{\mu} Y_{\mu}z{[\includegraphics[scale=.25]{Gamma12}]}$ to $T(G_2/e)$\\
$\{3\}$ & $\mathscr{D}$ &$x_{\mu} y_{\mu}^2z{[\includegraphics[scale=.25]{Gamma12}]}$ to $T_C$, $T(G_2/e)$\\
$\{1,2\}$ & $\mathscr{C}$ &
$x_{\mu}^2 y_{\mu}z{[\includegraphics[scale=.25]{Gamma12}]}$ to $T_L$, $X_{\mu}^2 y_{\mu}z{[\includegraphics[scale=.25]{Gamma12}]}$ to $T(G_2-e)$\\
$\{1,3\}$ & zero & $x_{\mu}^2
y_{\mu}z{[\includegraphics[scale=.25]{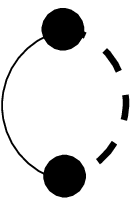}]}$ to $T_0$, $x_{\mu}^2
y_{\mu}z{[\includegraphics[scale=.25]{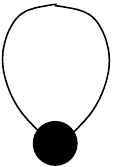}]}$ to $T(G_2/e)$, $x_{\mu}
y_{\mu}X_\mu z{[\includegraphics[scale=.25]{Gamma12}]}$ to $T(G_2-e)$\\
$\{2, 3\}$ & $\mathscr{D}$
& $x_{\mu}^2 y_{\mu}z{[\includegraphics[scale=.25]{Gamma11}]}$ to $T_C$, $x_{\mu}^2 Y_{\mu}z{[\includegraphics[scale=.25]{Gamma11}]}$ to $T(G_2/e)$\\
$\{1,2,3\}$ & $\mathscr{C}$
& $x_{\mu}^3z{[\includegraphics[scale=.25]{Gamma11}]}$ to $T_L$, $X_{\mu}x_{\mu}^2z{[\includegraphics[scale=.25]{Gamma11}]}$ to $T(G_2-e)$\\
\hline
\end{tabular}
\vspace{0.1in}
\caption{Contributions of the contracting sets to the pointed relative
  Tutte polynomials associated to  $G_2$.}\label{tab:contrib}
\end{center}
\end{table}
Summing up the weights listed in the table, we obtain
\begin{eqnarray*}
T(G_2-e)&=& x_\mu^2 X_\mu z{[\includegraphics[scale=.25]{Gamma11}]}+
y_\mu X_\mu(x_\mu+X_\mu)z{[\includegraphics[scale=.25]{Gamma12}]},\\
T(G_2/e)&=& x_\mu^2 (Y_\mu+y_\mu) z{[\includegraphics[scale=.25]{Gamma11}]}+ y_\mu(x_{\mu}y_{\mu}+x_\mu Y_\mu+ y_\mu X_\mu) z{[\includegraphics[scale=.25]{Gamma12}]},\\
T_0(G_2,e)&=&x_\mu^2y_\mu z{[\includegraphics[scale=.25]{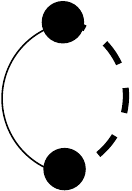}]}
\end{eqnarray*}
and
\begin{eqnarray*}
T_C(G_2,e)&=& x_\mu^2 y_\mu z{[\includegraphics[scale=.25]{Gamma11}]}+
y_\mu^2(2x_\mu+X_\mu)z{[\includegraphics[scale=.25]{Gamma12}]}\\
T_/(G_2,e)&=& T(G_2/e)-\pi_/T_0(G_2,e)\\
&=& x_\mu^2 Y_\mu z{[\includegraphics[scale=.25]{Gamma11}]}+ (X_{\mu}y_{\mu}^2+x_\mu
y_\mu Y_\mu+ x_\mu y_\mu^2) z{[\includegraphics[scale=.25]{Gamma12}]},\\
T_L(G_2,e)&=&x_\mu^3 z{[\includegraphics[scale=.25]{Gamma11}]}+ x_\mu^2 y_\mu z{[\includegraphics[scale=.25]{Gamma12}]},\\
T_-(G_2,e)&=& T(G_2-e)-\pi_-T_0(G_2,e)\\
&=&x_\mu^2 X_\mu  z{[\includegraphics[scale=.25]{Gamma11}]}+ (x_\mu X_\mu+X_\mu^2-x_\mu^2) y_\mu
z{[\includegraphics[scale=.25]{Gamma12}]}.
\end{eqnarray*}
We can now apply the mapping $\Phi=\beta_{0}\circ
\sigma\circ\beta_{\lambda}$ to $\sum_{S\subseteq E_{\lambda}(G_1)}
T^{\Lambda_0\cup\{\lambda_0\}}_{\H_1\cup
  S}(G_{1S})$ in an almost term by term fashion. The final answer
contains 7 different vertex pivot equivalent classes of graphs (with
only zero edges). For example,
to compute $\Phi(2x_\lambda z{[\includegraphics[scale=.25]{Gamma23}]})$, we
would replace $x_\lambda$ by $T_L(G_2,e)$, apply the regular splicing
map $\sigma$, and perform a two sum operation on the resulting graph
with $\Gamma_2^3$. This leads to
$$\Phi(2x_\lambda z{[\includegraphics[scale=.25]{Gamma23}]})= 2x_\mu^5 y_\mu  z{[\includegraphics[scale=.18]{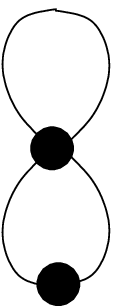}]}+2x_\mu^4
y_\mu^2 z{[\includegraphics[scale=.18]{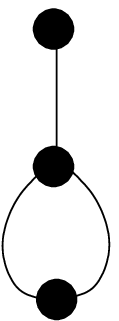}]}.$$
Similarly,
\begin{eqnarray*}
\Phi(x_\lambda^2 z{[\includegraphics[scale=.25]{Gamma11}]}) &=& x_{\mu}^6 z{[\includegraphics[scale=.18]{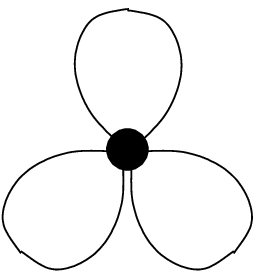}]} + 2 x_{\mu}^5
y_{\mu} z{[\includegraphics[scale=.18]{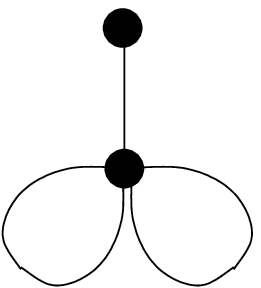}]} + x_{\mu}^4 y_{\mu}^2 z{[\includegraphics[scale=.18]{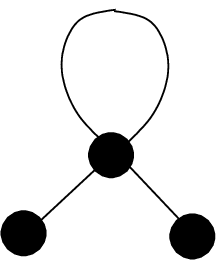}]},\\
\Phi(y_\lambda(x_\lambda + X_\lambda)z{[\includegraphics[scale=.25]{Gamma12}]})&=&
x_{\mu}^4 y_{\mu} (x_{\mu}+X_{\mu}) z{[\includegraphics[scale=.18]{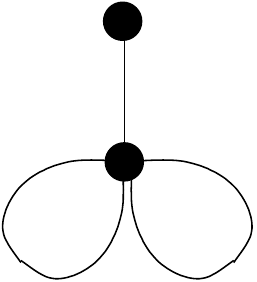}]}
+ x_{\mu}^2 y_{\mu}^2 (2 x_{\mu}^2 + 4 X_{\mu}x_{\mu}+2 X_{\mu}^2) z{[\includegraphics[scale=.18]{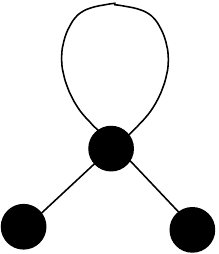}]} \\
&&+X_\mu y_{\mu}^3 (2x_{\mu}^2+3x_\mu X_\mu+X_{\mu}^2) z{[\includegraphics[scale=.18]{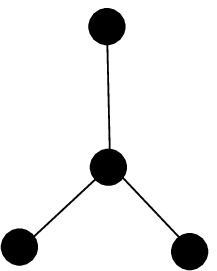}]},\\
\Phi(2y_\lambda z{[\includegraphics[scale=.2]{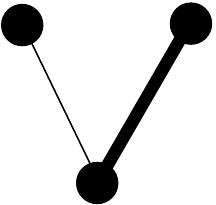}]}) &=& 2 x_{\mu}^4 y_{\mu}^2 z{[\includegraphics[scale=.18]{H3}]} +
2 x_{\mu}^2y_{\mu}^3 (2x_{\mu}+X_{\mu}) z{[\includegraphics[scale=.18]{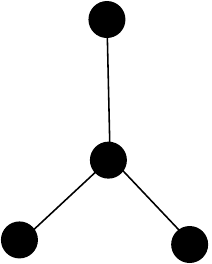}]}, \\
\Phi(z{[\includegraphics[scale=.2]{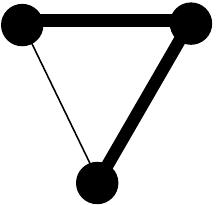}]}) &=& x_{\mu}^4 y_{\mu}^2 z{[\includegraphics[scale=.18]{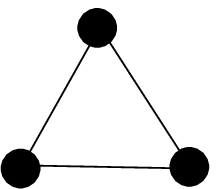}]}.\\
\end{eqnarray*}
We leave the verification of the details to our reader.  Summing up the
previous equations yields
\begin{eqnarray*}
T_\H^\Lambda(G_1\otimes_\lambda G_2)
&=&
x_\mu^6 z{[\includegraphics[scale=.18]{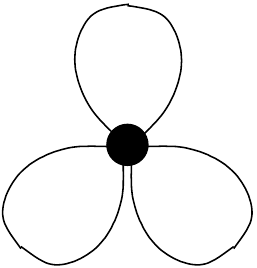}]}+x_\mu^4 y_\mu(3x_\mu+ X_\mu)z{[\includegraphics[scale=.18]{H2}]}+x_\mu^2 y_\mu^2(5x_\mu^2+4x_\mu X_\mu +2X_\mu^2)z{[\includegraphics[scale=.18]{H3}]}\\
&+&
y_\mu^3(4x_\mu^3+4x_\mu^2 X_\mu+3x_\mu X^2_\mu+X_\mu^3)z{[\includegraphics[scale=.18]{H4}]}+2x_\mu^5 y_\mu z{[\includegraphics[scale=.18]{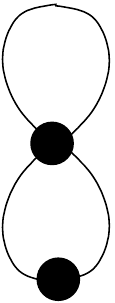}]}
+2x^4_\mu y_\mu^2 z{[\includegraphics[scale=.18]{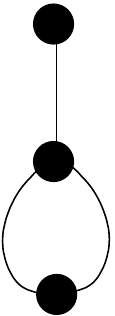}]}+x^4_\mu y_\mu^2 z{[\includegraphics[scale=.18]{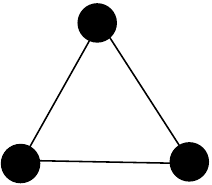}]}.
\end{eqnarray*}
As an exercise, we encourage our reader to verify this result by direct
contraction/deletion computation using the recursive formula
(\ref{recur1}), keeping in mind Definition~\ref{acti_alt_def}.

We end our paper with the following remark. In \cite{DHRel} we showed that the relative Tutte polynomial can be used to compute the Jones polynomial of a virtual knot, with a formulation very similar to the original work of Kauffman \cite{K2}. In the case of classical knot theory, our generalized formulation of the Tutte polynomial for a tensor product of colored graphs \cite{DGH1,DHCl} enables one to derive a fast computation of the Jones polynomial of a knot obtained through repeated tangle replacement operations \cite{DEZ}. Thus, the implication of our main result in virtual knot theory is that similar approaches are also possible for virtual knots obtained through repeated tangle replacement operations, so long as the tangle replacement does not occur at a virtual crossing (which corresponds to a zero edge in our graphs). A precise formulation and detailed analysis is, however, much more involved and is beyond the scope of this paper and shall be addressed by the authors in a future work.

\smallskip
\section*{Acknowledgement}
This work is supported in part by NSF Grants \#DMS-0920880 and \#DMS-1016460 to Y Diao.


\begin{thebibliography}{99}

\bibitem{BR} B.~Bollob\'as and O.~Riordan,
{\em A Tutte Polynomial for Coloured Graphs}, Combinatorics,
Probability and Computing \textbf{8} (1999), 45--93.

\bibitem{BR2} B.~Bollob\'as and O.~Riordan, {\em A polynomial of graphs
  on orientable surfaces}, Proc.\ London Math.\ Soc.\ {\bf 83} (2001),
  513--531.

\bibitem{BR3} B.~Bollob\'as and O.~Riordan, {\em A polynomial of graphs
  on surfaces}, Math.\ Ann.\ {\bf 323} (2002), 81--96.



\bibitem{B} T.\ Brylawski, {\em The Tutte polynomial I:
general theory}, in: {Matroid Theory and its
Applications}, ed. A. Barlotti, Liguori   Editore,
S.r.I, 1982, 125--275.

\bibitem{BO} T.\ Brylawski and J.\ Oxley, {\em The Tutte polynomial and its applications}, Matroid Applications (ed. N. White), Cambridge University Press (1992), 123--225.



\bibitem{Ch} S.~Chmutov, {\em Generalized duality for graphs on surfaces and the signed Bollob\'{a}s-Riordan polynomial}, J. Combin. Theory Ser. B \textbf{99} (2009), 617--638.

\bibitem{CP} S.~Chmutov and I.~Pak, {\em The Kauffman Bracket of
Virtual Links and the Bollo\'{a}s-Riordan Polynomial}, Moscow
Mathematical Journal \textbf{7}(3) (2007), 409--418.

\bibitem{CV} S.~Chmutov and J.~Voltz, {\em Thistlethwaite's theorem for virtual links}, J Knot Theory Ramification \textbf{17}(10) (2008), 1189--1198.

\bibitem{DEZ} Y.~Diao, C.~Ernst and U.~Ziegler, {\em Jones Polynomial of Knots formed by Repeated Tangle Replacement Operations}, Topology and its Applications \textbf{156} (2009), 2226--2239.

\bibitem{DHRel} Y.~Diao and G.~Hetyei, {\em Relative Tutte Polynomials for Colored Graphs
and Virtual Knot Theory}, Combinatorics, Probability
and Computing \textbf{19} (2010), 343--369.

\bibitem{DGH1} Y.~Diao, G.~Hetyei and K.~Hinson,
{\em Tutte Polynomials of Tensor Products of Signed Graphs and their
Applications in Knot Theory}, J Knot Theory Ramification \textbf{18}(5) (2009), 561--590.

\bibitem{DHCl} Y.~Diao, G.~Hetyei and K.~Hinson, {\em Invariants of
  composite networks arising as a tensor
product},  Graphs and Combinatorics \textbf{25} (2009), 273--290.



\bibitem{Ja} F.~Jaeger, D.~L.~Vertigan and D.~J.~A.~Welsh,
{\em On the Computational Complexity of the Jones and Tutte
Polynomials}, Math. Proc. Cambridge
Phil. Soc. \textbf{108} (1990), 35--53.



\bibitem{K2} L.~H.~Kauffman, {\em A Tutte Polynomial for
Signed Graphs}, Discrete Applied Mathematics \textbf{25} (1989),
105--127.

\bibitem{K3} L.~H.~Kauffman, {\em Virtual knot theory},  European J. Combin.  \textbf{20}(7)  (1999), 663--690.






\bibitem{T0} M.~B.~Thistlethwaite, {\em A Spanning Tree Expansion
for the Jones Polynomial}, Topology, \textbf{26} (1987), 297--309.



\end{thebibliography}
\end{document}